\documentclass[11pt,a4paper]{amsart}
\usepackage{amssymb,amsmath}
\usepackage{comment}

\def\({\left(}
\def\){\right)}

\def\Cal{\mathcal}

\def\eb{\varepsilon}

\def\R {\mathbb{R}}

\def\tilde{\widetilde}
\newcommand{\be}{\begin{equation} }
\newcommand{\ee}{\end{equation} }

\def\F {{\mathcal F}}

\def \and{\qquad\text{and}\qquad}

\def\Bbb{\mathbb}
\def\Dt{\partial_t}

\def\spann{\operatorname{span}}
\def\({\left(}
\def\){\right)}

\def\eb{\varepsilon}
\def\Cal{\mathcal}
\def\eb{\varepsilon}

\def\R {\mathbb{R}}

\def\F {{\mathcal F}}

\def\<{\left<}
\def\>{\right>}

\def \and{\qquad\text{and}\qquad}

\def\Bbb{\mathbb}
\def\Dt{\partial_t}

\newtheorem{proposition}{Proposition}[section]
\newtheorem{theorem}[proposition]{Theorem}
\newtheorem{corollary}[proposition]{Corollary}
\newtheorem{lemma}[proposition]{Lemma}
\theoremstyle{definition}
\newtheorem{definition}[proposition]{Definition}
\newtheorem{remark}[proposition]{Remark}
\newtheorem{example}[proposition]{Example}
\numberwithin{equation}{section}

\def\be{\begin{equation}}
\def\ee{\end{equation}}
\def\bp{\begin{proof}}
\def\ep{\end{proof}}
\def \au {\rm}

\def \bk {\it}
\def \no#1#2#3 {{\bf #1} (#3), #2.}
\def \eds#1#2#3 {#1, #2, #3.}

\title[Extensions of inertial manifolds]
{Smooth extensions for inertial manifolds of semilinear parabolic equations}
\author[A. Kostianko  and  S. Zelik]
{ Anna Kostianko and Sergey Zelik${}^{1,2}$}

\address{${}^1$
Department of Mathematics,\newline \indent
University of Surrey, GU27XH,
Guildford,  UK}

\address{${}^2$ \phantom{e}School of Mathematics and Statistics, Lanzhou University, Lanzhou
\newline\indent 730000,
P.R. China}
\email{anna.kostianko@surrey.ac.uk}
\email{s.zelik@surrey.ac.uk}

\begin{document}

\begin{abstract} The paper is devoted to a comprehensive study of  smoothness  of inertial manifolds for abstract semilinear parabolic problems. It is well known that in general we cannot expect more than $C^{1,\eb}$-regularity for such manifolds (for some positive, but small $\eb$). Nevertheless, as shown in the paper, under the natural assumptions, the obstacles to the existence of a $C^n$-smooth inertial manifold (where $n\in\Bbb N$ is any given number) can be removed by increasing the dimension and by modifying properly the nonlinearity outside of the global attractor (or even outside the $C^{1,\eb}$-smooth IM of a minimal dimension). The proof is strongly based on the Whitney extension theorem.
\end{abstract}

\thanks{This work is partially supported by  the RSF grant   19-71-30004  as well as  the EPSRC
grant EP/P024920/1. The authors also would like to thank Dmitry Turaev for many fruitful discussions.}
\maketitle
\tableofcontents

\section{Introduction}
It is believed that in many cases the long-time behaviour of infinite dimensional dissipative dynamical systems generated   by evolutionary PDEs (at least in bounded domains)   can be effectively described by finitely many parameters (the so-called order parameters in the terminology of I. Prigogine) which obey a system of ODEs. This system of ODEs (if exists) is usually referred as an Inertial Form (IF) of the considered PDE, see \cite{ha,Rob1,R11,T97,Z14} and references therein for more details.  However, despite the fundamental significance of this reduction from both theoretical and applied points of view and big interest during the last 50 years, the nature of such a reduction and its rigorous justification remains a mystery.
\par
Indeed, it is well  understood now that the key question of the theory is how smooth the desired IF can/should be. For instance, in the case of {\it H\"older} continuous IFs, there is a highly developed machinery for constructing them based on the so-called attractors theory and the Mane projection theorem. We recall that, by definition, a global attractor is a compact invariant set in the phase space of the dissipative system considered which attracts as time goes to infinity the images of bounded sets under the evolutionary semigroup related with the considered problem. Thus, on the one hand, a global attractor (if it exists) contains all of the non-trivial dynamics and, on the other hand, it is usually essentially "smaller" than the initial phase space and this second property allows us to speak about the reduction of degrees of freedom in the limit dynamics. In particular, one of the main results of the attractors theory tells us that, under relatively weak assumptions on a dissipative PDE (in a bounded domain), the global attractor exists and has finite Hausdorff and fractal dimensions. In turn, due to the Mane projection theorem, this finite-dimensionality guarantees that this attractor can be projected one-to-one to a generic finite-dimensional plane of the phase space and that the inverse map is H\"older continuous. Finally, this scheme gives us an IF with H\"older continuous vector field defined on some compact set of $\R^N$ which is treated as a rigorous justification of the above mentioned finite-dimensional reduction. This approach works, for instance, for 2D Navier-Stokes equations, reaction-diffusion systems, pattern formation equations, damped wave equations, etc., see \cite{BV92,3,CV02,ha,hen,hunt,MZ08,R11,SY02,T97} and references therein.
\par
However, the above described scheme has a very essential intrinsic drawback which prevents us to treat it as a satisfactory solution of the finite-dimensional reduction problem. Namely, the vector field in the  IF thus constructed is H\"older continuous {\it only} and there is no way in general to get even its Lipschitz continuity. As a result, we may lose the uniqueness of solutions for the obtained IF and have to use the initial infinite-dimensional system at least in order to select the correct solution of the reduced IF. Other drawback is that the Mane projection theorem is not constructive, so it is not clear how to choose this "generic" plane for projection in applications, in addition, the IF constructed in such a way is defined only on a complicated compact set (the image of the attractor under the projection) and it is not clear how to extend it on the whole $\R^N$ preserving the dynamics (surprisingly, this is also a deep open problem, some partial solution of it is given in \cite{Rob1}, see also the references therein).
\par
It is also worth noting that the restriction for IF to be only H\"older continuous is far from being just a technical problem here. As relatively simple counterexamples show (see \cite{EKZ13,KZ18,MPSS93,Rom00,Z14}) the fractal dimension of the global attractor may be finite
 and not big, but the attractor cannot be embedded into any finite-dimensional Lipschitz (or even $\log$-Lipschitz) finite-dimensional sub-manifold of the phase space. What is even more important, the dynamics on this attractor does not look as finite-dimensional at all (despite the existence of a {\it H\"older continuous} (with the H\"older exponent arbitrarily close to one) IF provided by the Mane projection theorem). For instance, it may contain limit cycles with super-exponential rate of attraction, decaying travelling waves in Fourier space and other phenomena which are impossible in the classical dynamics generated by smooth ODEs.  These examples suggest that, in contradiction to the widespread paradigm,  H\"older continuous IF is probably not an appropriate tool for distinguishing between finite and infinite dimensional limit behavior and, as a result, fractal-dimension is not so good for estimating the number of degrees of freedom for the reduced dynamics, see \cite{EKZ13,KZ18,Z14} for more details.
\par
An alternative, probably more transparent approach to the finite-di\-men\-sio\-nal reduction problem which has been suggested in \cite{FST88}  is related with the concept of an Inertial Manifold (IM). By definition, an IM is a finite-dimensional smooth (at least Lipschitz) invariant sub-manifold of the phase space which is globally exponentially stable and possesses the so called exponential tracking property (=existence of asymptotic phase). Usually this manifold is $C^{1,\eb}$-smooth for some positive $\eb$ and is
 normally hyperbolic, so the exponential tracking is an immediate corollary of normal hyperbolicity. Then the corresponding IF is just a restriction of the initial PDE to IM and is also $C^{1,\eb}$-smooth. However, being a sort of center manifold, an IM requires a separation of the dependent variable to the "slow" and "fast" components and this, in turn, leads to extra rather restrictive assumptions which are usually formulated in terms of {\it spectral gap} conditions. Namely, let us consider the following abstract semilinear parabolic equation in a real Hilbert space $H$:
 \begin{equation}\label{00.abs}
 \partial_t u+Au=F(u),\ \ u\big|_{t=0}=u_0,
 \end{equation}
 where $A: D(A)\to H$ is a self-adjoint positive operator such that $A^{-1}$ is compact and $F:H\to H$ is a given nonlinearity which is globally Lipschitz in $H$ with Lipschitz constant $L$. Let also $0<\lambda_1\le\lambda_2\le\cdots$ be the eigenvalues of $A$ enumerated in the
 non-decreasing order  and $\{e_n\}_{n=1}^\infty$ be the corresponding eigenvectors. Then, the sufficient condition for the existence of $N$-dimensional IM reads
 \begin{equation}\label{00.sg}
 \lambda_{N+1}-\lambda_N>2L.
 \end{equation}
If this condition is satisfied, the desired IM $\Cal M_N$ is actually a graph of a Lipschitz function $M_N: H_N\to (H_N)^\bot$, where $H_N=\spann\{e_1,\cdots, e_N\}$ is a spectral subspace spanned by the first $N$ eigenvectors, and the corresponding IF has the form
\begin{equation}\label{00.if}
\frac {d}{dt}u_N+Au_N={\bf P}_NF(u_N+M_N(u_N)),\ \ u_N\in H_N\sim\R^N,
\end{equation}
where ${\bf P}_N$ is the orthoprojector to $H_N$, see \cite{S-sell,CFNT89,FST88,kok,M91,R94,RT96, Z14} and also \S\ref{s1} below.
\par
We see that, in contrast to the IF constructed via the Mane projection theorem, the IF which corresponds to the IM is explicit (uses the spectral projections) and is as smooth as the functions $F$ and $M_N$ are. We  mention that although the spectral gap condition \eqref{00.sg} is rather restrictive (e.g. in the case where $A$ is a Laplacian in a bounded domain, it is satisfied in 1D case only) and is known to be sharp in the class of abstract semilinear parabolic equations (see \cite{EKZ13,M91,R94,Z14} for more details), it can be relaxed for some concrete classes of PDEs. For instance,  for scalar 3D reaction-diffusion equations (using the so-called spatial averaging principle, see \cite{M-PS88}), for 1D reaction-diffusion-advection systems (using the proper integral transforms, see \cite{KZ18,KZ17}), for 3D Cahn-Hilliard equations and various modifications of   3D Navier-Stokes equations (using various modifications of spatial-averaging, see \cite{GG18,K18,KZ15,LS20}), for 3D complex Ginzburg-Landau equation (using the so-called spatio-temporal averaging, see \cite{K20}), etc. Note also that the global Lipschitz continuity assumption for the non-linearity $F$ is not an essential extra restriction since usually one proves the well-posedness and dissipativity of the PDE under consideration {\it before} constructing the IM. Cutting off the non-linearity outside the absorbing ball does not affect the limit dynamics, but  reduces the case of locally Lipschitz continuous non-linearity (satisfying the proper dissipativity restrictions) to the model case where the non-linearity is globally Lipschitz continuous. Of course, this cut-off procedure is not unique and as we will see below, the right choice of it is extremely important in the theory of IMs.
\par
The main aim of the present paper is to study the smoothness of the IFs for semilinear parabolic equations \eqref{00.abs} in the ideal situation where the non-linearity $F$ is smooth and the
spectral gap condition \eqref{00.sg} is satisfied. As we have already mentioned, in this case we have $C^{1,\eb}$-smooth IM $\Cal M_N$ for some $\eb>0$ and the associated IF \eqref{00.if} which is also $C^{1,\eb}$-smooth, see \cite{Z14} and references therein. But, unfortunately, the exponent $\eb>0$ here is usually very small (depending on the spectral gap) and in a more or less general situation, we cannot expect even the $C^2$-regularity of the IM. The spectral gap condition for $C^2$-regular IM is
\begin{equation}\label{00.sg2}
\lambda_{N+1}-2\lambda_N>3L
\end{equation}
and such exponentially big spectral gaps are not available if $A$ is a finite order elliptic operator in a bounded domain. The corresponding counterexamples were given in \cite{S-sell}, see also Example \ref{Ex2.sell} below. Thus, the existing IM theory does not allow us, even in the ideal situation, to construct more regular than $C^{1,\eb}$ IFs (where $\eb>0$ is small). This looks as an essential drawback at least by two reasons: 1) The lack of regularity prevents us to use higher order methods for numerical simulations of the reduced IF (as a result, direct simulations for the initial {\it smooth} PDE using the standard methods may be more effective than simulations based on the reduced non-smooth ODEs); 2) $C^{1,\eb}$-regularity is not enough to build up normal forms and/or study the bifurcations properly (for instance, the simplest saddle-node  bifurcation requires $C^2$-smoothness, the Hopf bifurcation needs $C^3$, etc., see \cite{katok,kie} for more details) and, therefore, we need to return back to the initial PDE to study these bifurcations.  Thus, the natural question
\par
"{\it Is it possible to construct a smooth ($C^k$-smooth for any finite $k$) or to extend the existing $C^{1,\eb}$-smooth IF to a more regular one}?"
\par
become crucial for the theory of inertial manifolds.
\par
In the present paper we give an affirmative answer on this question under slightly stronger spectral gap assumption
\begin{equation}\label{00.sginf}
\limsup_{N\to\infty}(\lambda_{N+1}-\lambda_N)=\infty.
\end{equation}
In contrast to \eqref{00.sg2}, this assumption does not require exponentially big spectral gaps (and is satisfied for the  most part of examples where the IMs exist), but guarantees the existence of infinitely-many spectral gaps of size larger than $2L$ and, consequently, the existence of an infinite tower of the embedded IMs:
\begin{equation}\label{00.tim}
\Cal M_{N_1}\subset\Cal M_{N_2}\subset\cdots\subset\Cal M_{N_n}\subset\cdots
\end{equation}
and the corresponding IFs
\begin{equation}\label{00.tif}
\frac d{dt}u_{N_n}+Au_{N_n}={\bf P}_{N_n}F(u_{N_n}+M_{N_n}(u_{N_n})),\ \ u_{N_n}\in H_{N_n}.
\end{equation}
Let $n\in\Bbb N$ be given. We say that a $C^{n,\eb}$-smooth submanifold $\widetilde{\Cal M}_{N_n}$ of the phase space $H$ (which is a graph of $C^{n,\eb}$-smooth $\widetilde M_{N_n}: H_{N_n}\to  (H_{N_n})^\bot$)  is a $C^{n,\eb}$-smooth extension of the initial IM $\Cal M_{N_1}$ for some $\eb>0$ if
\par
1) $\Cal M_{N_1}\subset\widetilde{\Cal M}_{N_n}$;
\par
2) The manifold $\widetilde{\Cal M}_{N_n}$ is $C^1_b$-close to the IM $\Cal M_{N_n}$.
\par
Then, the first condition guarantees that the $C^{n,\eb}$-smooth system of ODEs
\begin{equation}\label{00.ifm}
\frac d{dt}u_{N_n}+Au_{N_n}={\bf P}_{N_n}F(u_{N_n}+\widetilde M_{N_n}(u_{N_n})), \ \ u_{N_n}\in H_{N_n}
\end{equation}
will possess the initial IM ${\bf P}_{N_n}\Cal M_{N_1}$ as an invariant submanifold. The second condition together with  the robustness theorem for normally hyperbolic manifolds ensures us that this manifold will be globally exponentially stable and normally hyperbolic (in particular, it will possess an exponential tracking property in $H_{N_n}$). In this case we refer to the system \eqref{00.ifm} as a $C^{n,\eb}$-smooth extension of the corresponding IF \eqref{00.if}, see \S3 for more details. Thus, the extended IF is $C^n$-smooth on the one hand and, on the other hand, its limit dynamics {\it coincides} with the dynamics of the IF which corresponds to the IM $\Cal M_{N_1}$ and, in turn, coincides with the limit dynamics of the initial abstract parabolic problem \eqref{00.abs}. Note that the manifold $\widetilde{\Cal M}_{N_n}$ is {\it not necessarily} invariant under the solution semigroup $S(t)$ generated by the initial equation \eqref{00.abs} and this allows us to overcome the standard obstacles to the smoothness of an invariant manifold (e.g. such as resonances, see Examples \ref{Ex2.sell} and  \ref{Ex6.sell} below).
\par
The main result of the paper is the following theorem which suggests a solution of the smoothness problem for IMs.
\begin{theorem}\label{Th00.main} Let the nonlinearity $F\in C^\infty_b(H,H)$ and let the operator $A$ satisfy the spectral gap conditions \eqref{00.sginf}. Let also $N_1\in\Bbb N$ be the smallest number for which the spectral gap condition \eqref{00.sg} is satisfied and $\Cal M_{N_1}$ be the corresponding IM. Then, for every $n\in \Bbb N$, there exists $\eb=\eb_n>0$ and the $C^{n,\eb}$-smooth extension of
the IM $\Cal M_{N_1}$ as well as the $C^{n,\eb}$-smooth extension of the corresponding IF in the sense described above.
\end{theorem}
The proof of this theorem is given in \S\ref{s3} and \S\ref{s4}. To construct the desired extension $\widetilde M_{N_n}$, we first define it on the manifold ${\bf P}_{N_n}\Cal M_{N_1}$ only in a natural way $\widetilde M_{N_n}(p)=(1-{\bf P}_{N_n})M_{N_1}({\bf P}_{N_1}p)$. Then, we present an explicit construction of Taylor jets of order $n$ for this function via some inductive procedure, see \S\ref{s3}. Finally, we check (in \S\ref{s4}) the compatibility conditions for the constructed Taylor jets and get the desired extension by the Whitney extension theorem.
\par
Our main result can be reformulated in the following way.
\begin{corollary} \label{Cor00.main} Let the assumptions of Theorem \ref{Th00.main} hold. Then, for every $n\in\Bbb N$, there exists $\eb=\eb_n>0$ and $C^{n-1,\eb}$-smooth "correction"  $\widetilde F_n(u)$ of the initial nonlinearity  $F$ such that
\par
1) $\widetilde F_n(u)=F(u)$ for all $u\in\Cal M_{N_1}$ and $\Cal M_{N_1}$ is an IM for the modified equation
\begin{equation}\label{00.absM}
\partial_t u+Au=\widetilde F_n(u),\ \ u\big|_{t=0}=u_0
\end{equation}
as well. In particular, the dynamics of \eqref{00.absM} on $\Cal M_{N_1}$ coincides with the initial dynamics (generated by \eqref{00.abs}) and $\Cal M_{N_1}$ possesses an exponential tracking property for solutions of \eqref{00.absM}.
\par
2) The extended manifold $\widetilde {\Cal M}_{N_n}$ constructed in Theorem \ref{Th00.main} is an IM (of smoothness $C^{n,\eb}$) for the modified equation  \eqref{00.absM}, see Corollary \ref{Cor5.eq} below.
\end{corollary}
In this interpretation, the modified nonlinearity $\widetilde F_n$ can be considered as a "cutted-off" version of the initial function $F$ and the main result claims that all obstacles for the existence of $C^n$-smooth IM can be removed by increasing the dimension of the IM and using the properly chosen cut-off procedure.
\par
To conclude, we note that the main aim of this paper is to verify the principal possibility to get smooth extensions of IM rather than to obtain the optimal bounds for the dimensions $N_n$ of the constructed extensions. By this reason, the obtained bounds look far from being optimal, but we believe that they can be essentially improved, see Remark \ref{Rem5.final} for the discussion of this problem.
\par
The paper is organized as follows. In \S\ref{s1} we recall the standard facts about smooth functions in Banach spaces, their Taylor jets, direct and converse Taylor theorems and the Whitney extension theorem which is the main technical tool for what follows. In \S\ref{s2} we collect basic facts about the construction of IMs for semilinear parabolic equations via the Perron method and discuss known facts about the smoothness of these IMs. The main result (Theorem \ref{Th00.main}) is presented in \S\ref{s3}. The proof of it is also given there by modulo of compatibility conditions for Whitney extension theorem which are verified in \S\ref{s4}. Finally, the applications of the proved theorem as well as a discussion of open problems and related topics are given in \S\ref{s5}.

\section{Preliminaries I: Taylor expansions and Whitney Extension Theorem}\label{s1}

In this section we briefly recall the standard results on Taylor expansions
of smooth functions in Banach spaces and related Whitney extension theorem
 as well as prepare some technical tools which will be used later. We
start with some basic facts from multi-linear algebra, see e.g. \cite{hajo} for a more detailed exposition.
Let $X$ and  $Y$ be two normed spaces. For any $n\in\Bbb N$, we denote by $\Cal L_s(X^n,Y)$
the space of multi-linear continuous symmetric maps from $X^n$ to $Y$ endowed
by the standard norm
$$
\|M\|_{\Cal L_s(X^n,Y)}:= \sup_{\xi_i\in X,\, \xi_i\ne0}
\left\{\frac{\|M(\xi_1,\cdots,\xi_n)\|}{\|\xi_1\|\cdots\|\xi_n\|}\right\}.
$$
Every element $M\in\Cal L_s(X^n,Y)$ defines a homogeneous continuous polynomial
$P_M$ of order $n$ on $X$ with values in $Y$ via
$$
P_M(\xi):= M(\{\xi\}^n),\ \ \text{ where } \{\xi\}^n:=\underbrace{\xi,\cdots,\xi}_{\text{$n$-times}}.
$$

Vice versa, the multi-linear symmetric map $M=M_P$ can be restored in a
unique way if the corresponding homogeneous polynomial is known via the
polarization equality:
$$
M_P (\xi_1,\cdots,\xi_n)=\frac{1}{2^nn!}\sum_{\eb_i=\pm1, i=1,\cdots,n}\eb_1\cdots\eb_n
 P(a+\sum_{j=1}^n\eb_j\xi_j)
$$
for all $a,\xi_1,\cdots,\xi_n\in X$, see e.g. \cite{hajo}. Thus, there is a one-to-one correspondence
between homogeneous polynomials and multi-linear symmetric
maps. Moreover, if we introduce the following norm on the space $\Cal P_n(X,Y)$
of $n$-homogeneous polynomials
$$
\|P\|_{\Cal P_n(X,Y)}:=\sup_{\xi\ne0}\left\{\frac{\|P(\xi)\|}{\|\xi\|^n}\right\},
$$
this correspondence becomes an isometry. By this reason, we will identify
below multi-linear forms and the corresponding homogeneous polynomials
where this does not lead to misunderstandings. We also mention
here the generalization of the Newton binomial formula, namely, for any
$P\in\Cal P_n(X,Y)$ and $\xi,\eta\in X$, we have
\begin{equation}\label{2.1}
 P(\xi+\eta)=\sum_{j=0}^n C_n^j P(\{\xi\}^j,\{\eta\}^{n-j}),\ \  C^j_n:=\frac{n!}{j!(n-j)!},
\end{equation}
see e.g. \cite{hajo}. Finally, we denote by $\Cal P^n(X,Y)$ the space of all continuous
polynomials of order less than or equal to $n$ on $X$ with values in $Y$, i.e. $P(\xi)\in\Cal P^n(X,Y)$ if
$$
P(\xi) =\sum_{j=0}^n\frac1{j!}P_j(\xi),\ \  P_j(\xi)\in \Cal P_j(X,Y).
$$
The following standard result is crucial for our purposes.
\begin{lemma}\label{Lem2.1} For every  $n\in\Bbb N$ there exist real numbers
$a_{kj}\in\R$, $k,j\in \{0,\cdots,n\}$,
such that for every $P=\sum_{k=0}^n\frac1{k!}P_k$, $P_k\in\Cal P_k(X,Y)$ and every
$k\in\{0,\cdots,n\}$, we have
\begin{equation}\label{2.2}
P_k(\xi) =\sum_{j=0}^na_{kj}P\(\frac jn\xi\)
\end{equation}
and, therefore,
\begin{equation}\label{2.3}
 \|P_k(\xi)\|\le K_{n,k} \max_{j=0,\cdots, n}\|P\(\frac jn\xi\)\|
\end{equation}
for some constants $K_{n,k}$ which are independent of $P$.
\end{lemma}
For the proof of this lemma see \cite{hajo}.

\begin{corollary}\label{Cor2.2} Let $P(\xi, \delta)\in P^n(X,Y)$ be a family of polynomials of $\xi$
depending on a parameter $\delta\in B$ where $B$ is a  set in $X$ containing
zero. Assume that
\begin{equation}\label{2.4}
 \|P(\xi,\delta)\| \le C(\|\xi\|+\|\delta\|)^{n+\alpha},\ \  \xi\in X,\ \  \delta\in B
\end{equation}
for some $\alpha\ge0$. Then, for any $k\in\{0,\cdots,n\}$,
\begin{equation}\label{2.5}
 \|P_k(\cdot,\delta)\|_{\Cal P_k(X,Y)}\le C_k\|\delta\|^{n-k+\alpha}
\end{equation}
for some constants $C_k$ depending on $C$, $n$ and $k$.
\end{corollary}
\begin{proof} Indeed, according to \eqref{2.3} and \eqref{2.4}, we have
$$
\|P_k(\xi,\delta)\|\le C'(\|\xi\|+\|\delta\|)^{n+\alpha}.
$$
Assuming that $\delta\ne0$ (there is nothing to prove otherwise), replacing $\xi$ by
$\|\delta\|\xi$ and using that $P_k$ is homogeneous of order $k$, we get
$$
\|P_k(\xi,\delta)\|\le C'(1 + \|\xi\|)^{n+\alpha}\|\delta\|^{n-k+\alpha}.
$$
Using once more that $P_k$ is homogeneous of order $k$ in $\xi$, we finally arrive at
$$
\|P_k(\xi,\delta)\|\le C''\|\xi\|^k\|\delta\|^{n-k+\alpha}
$$
which gives \eqref{2.5} and finishes the proof.
\end{proof}

Let  now $U \subset X$ be an open set and let $F:U\to Y$ be a map. As usual, for any
$u\in U$, we denote by $F'(u)\in \Cal L(X,Y)$ the Frechet derivative of $F$ at $u$ (if
it exists). Analogously, for any $n\in\Bbb N$, we denote by $F^{(n)}(u)\in \Cal L_s(X^n,Y)$
its $n$th Frechet derivative. The space of all functions $F: U\to Y$ such that
$F^{(n)}(u)$ exists and continuous as a function from $U$ to $\Cal L_s(X^n,Y)$ is denoted
by $C^n(U,Y)$. For any $\alpha\in(0,1]$, we denote by $C^{n,\alpha}(U,Y)$ the space of
functions $F\in C^n(U,Y)$ such that $F^{(n)}$ is H\"older continuous with exponent
$\alpha$ on $U$. The action of $F^{(n)}(u)$ to vectors $\xi_1,\cdots,\xi_n\in X$ is denoted by
$F^{(n)}(u)[\xi_1,\cdots,\xi_n]$. The Taylor jet of length $n+1$ of the function $F$ at point
$u$ and vector $\xi\in X$ will be denoted by $J^n_\xi F(u)$:
\begin{equation}\label{2.6}
 J^n_\xi F(u):=F(u)+\frac1{1!}F'(u)\xi+\frac1{2!}F''(u)[\xi,\xi]+\cdots+\frac1{n!}F^{(n)}(u)[\{\xi\}^n].
\end{equation}
Obviously, the function $\xi\to J^n_\xi F(u)\in\Cal P^n(X,Y)$ for every $u\in U$. We will
also systematically use the truncated Taylor jets
\begin{equation}\label{2.7}
 j^n_\xi F(u):=\frac1{1!}F'(u)\xi+\frac1{2!}F''(u)[\xi,\xi]+\cdots+\frac1{n!}F^{(n)}(u)[\{\xi\}^n]
 \end{equation}
which do not contain zero order term.

\begin{theorem}[Direct Taylor theorem]\label{Th2.3} Let $F\in C^n(U,Y)$ and  $u_1,u_2\in U$
be such that $u_t:=tu_1+(1-t)u_2\in U$ for all $t\in[0,1]$. Let also $\xi:=u_2-u_1$.
Then
\begin{multline}\label{2.8}
F(u_2)=J^n_\xi F(u_1)+\\+\frac1{n!}\int^1_0(1-s)^{n-1}\(F^{(n)}(u_1+s\xi)-F^{(n)}(u_1)\)\,ds [\{\xi\}^n].
\end{multline}
In particular, if $F\in C^{n,\alpha}(U,Y)$, then
\begin{equation}\label{2.9}
 \|F(u_2)-J^n_\xi F(u_1)\|\le C\|\xi\|^{n+\alpha}
 \end{equation}
for some positive $C$.
\end{theorem}
For the proof of this classical result see e.g. \cite{hajo}. We also mention that
in terms of truncated jets formula \eqref{2.9} reads
\begin{equation}\label{2.10}
 F(u_2)-F(u_1)=j^n_\xi F(u_1)+O(\|\xi\|^{n+\alpha}),\ \  \xi:=u_2-u_1.
 \end{equation}
The above theorem can be inverted as follows.

\begin{theorem}[Converse Taylor theorem]\label{Th2.4} Let function $F$ be such that, for
any $u\in U$ there exists a polynomial $\xi\to P(\xi,u)\in \Cal P^n(X,Y)$ such that, for
all $u_1,u_2\in U$,
\begin{equation}\label{2.11}
\|F(u_2)-P(\xi,u_1)\|\le C\|\xi\|^{n+\alpha},\ \ \xi:=u_2-u_1
\end{equation}
for some $C>0$ and $\alpha\in(0,1]$. Then, $F\in C^n(U,Y)$,
$$
P(\xi, u)=J^n_\xi F(u)
$$
for all $u\in U$ and $F^{(n)}(u)$ is locally H\"older continuous in $U$
with exponent $\alpha$. If, in addition, $U$ is convex, then $F\in C^n(U,Y)$ and
$$
\|F^{(n)}(u_2)-F^{(n)}(u_1)\|\le C'\|u_2-u_1\|^\alpha,
$$
where $C'$ depends only on $n$, $\alpha$ and the constant $C$ from \eqref{2.11}.
\end{theorem}
For the proof of this theorem see \cite{hajo}.
\par
Keeping in mind the Whitney extension problem, we recall that arbitrarily
chosen set of polynomials $P(\xi,u)$, $u\in U$, does not define in general a $C^{n,\alpha}$-
smooth function, but some compatibility conditions must be satisfied for
that. Indeed, let $u_1\in U$ and let $\delta,\xi\in X$ be such that $u_2:=u_1+\delta\in U$ and
$u_3:=u_1+\delta+\xi=u_2+\xi\in X$. Then, from \eqref{2.9}, we have
$$
\|F(u_3)-P(\xi+\delta, u_1)\|\le C\|\xi+\delta\|^{n+\alpha}
$$
and
$$
\|F(u_3)-P(\xi,u_1+\delta)\|\le C\|\xi\|^{n+\alpha}.
$$
Therefore,
\begin{equation}\label{2.12}
 \|P(\xi+\delta,u_1)-P(\xi, u_1+\delta)\|\le C_1(\|\xi\|+\|\delta\|)^{n+\alpha}.
\end{equation}
These are the desired compatibility conditions. In other words, if we are
given a set $V\subset X$ and a family of polynomials
$$
\{P(\xi,u),\  u\in V\}\subset P^n(X,Y)
$$
and want to find a function $F\in C^{n,\alpha}(X,Y)$ such that $J^n_\xi F(u)=P(\xi,u)$
for all $u\in V$, then the compatibility conditions \eqref{2.12} must be satisfied for all
$u_1,u_1+\delta\in V$ and all $\xi\in X$.
\par
Inequalities \eqref{2.12} can be rewritten in a more standard form which usually
appears in the statement of Whitney extension theorem. Namely, using \eqref{2.1}, we see that
\begin{equation*}
P(\xi+\delta,u_1)=\sum_{l=0}^n\frac1{l!}\sum_{k=l}^{n}\frac1{(k-l)!} P_k([\{\xi\}^l, \{\delta\}^{k-l}],u_1)
\end{equation*}
where $P(\xi,u_1)=\sum_{l=0}^n\frac1{l!}P_l([\{\xi\}^l], u_1)$, $P_l(\cdot,u_1)\in\Cal P_l(X,Y)$.
Applying now
Corollary \ref{Cor2.2} to \eqref{2.12}, we get the desired alternative form of the compatibility
conditions:
\begin{equation}\label{2.13}
 \|P_l(\{\xi\}^l,u_1+\delta)-\sum^{n-l}_{k=0}\frac 1{k!}P_{l+k}([\{\xi\}^l, \{\delta\}^k], u_1)\|\le
 C\|\xi\|^l\|\delta\|^{n-l+\alpha}
 \end{equation}
for $l = \{0,\cdots,n\}$. Compatibility conditions \eqref{2.13} have natural interpretation:
if $P_k(\{\xi\},u_1)=F^{(k)}(u_1)[\{\xi\}^k]$ as we expect, then \eqref{2.13} is nothing else than
Taylor expansions of $F^{(l)}(u_1+\delta)[\{\xi\}^l]$ at $u_1$.
\par
The next theorem shows that the introduced compatibility conditions are
sufficient for the existence of $F$ in the case when $X$ is finite-dimensional.
\par
\begin{theorem}[Whitney extension theorem]\label{Th2.5} Let $\dim X<\infty$ and let $V$
be an arbitrary subset of $X$. Assume also that we are given a family of
polynomials $\{P(\xi,u),\ u\in V\}\subset P^n(X,Y)$ which satisfies the compatibility
conditions \eqref{2.12} with some $\alpha\in(0,1]$. Then, there exists a function
 $F\in C^{n,\alpha}(X,Y)$ such that $J^n_\xi F(u)=P(\xi,u)$ for all $u\in V$.
 \end{theorem}
For the proof of this theorem see \cite{stein} or \cite{fef}.
 Note that the theorem fails if the
dimension of $X$ is infinite, but there are no restrictions on the dimension of
the space $Y$, see \cite{wells}.

\section{Preliminaries II: Spectral gaps and the construction of an
inertial manifold}\label{s2}

In this section we briefly discuss the classical theory of inertial manifolds
for semilinear parabolic equations, see e.g. \cite{Z14} for a more detailed exposition.
\par
Let $H$ be an infinite-dimensional real Hilbert space. Let us consider an
abstract parabolic equation in $H$:
\begin{equation}\label{3.1}
\Dt u+Au=F(u),\ \ u\big|_{t=0}=u_0,
\end{equation}
where $A:D(A)\to H$ is a linear self-adjoint positive operator in $H$ with
compact inverse and $F\in C^\infty_b(H,H)$ is a smooth bounded function on $H$
such that all its derivatives are also bounded on $H$.
\par
It is well-known that under the above assumptions equation \eqref{3.1} is globally
well-posed for any $u_0\in H$ in the class of solutions $u\in C([0,T],H)$ for
all $T>0$ and, therefore, generates a semigroup in $H$:
\begin{equation}\label{3.2}
S(t):H\to H,\ \  t\ge0,\ \  S(t)u_0:=u(t).
\end{equation}
Moreover, the solution operators $S(t)\in C^\infty(H,H)$ for every fixed $t\ge0$, see
\cite{hen,Z14} for the details.
\par
Let $0 < \lambda_1\le\lambda_2\le\cdots$ be the eigenvalues of the operator $A$ enumerated in the
non-decreasing order and let $\{e_n\}_{n=1}^\infty$
be the corresponding orthonormal system of eigenvectors. Then, by the Parseval
equality, for every $u\in H$, we have
$$
\|u\|^2_H =\sum_{n=1}^\infty(u,e_n)^2,\ \  u=\sum_{n=1}^\infty(u,e_n)e_n,
$$
where $(\cdot,\cdot)$ is an inner product in $H$. For a given $N\in\Bbb N$, we denote by ${\bf P}_N$
and ${\bf Q}_N$ the orthoprojectors on the first $N$ and the rest of eigenvectors of $A$
respectively:
$$
{\bf P}_Nu :=\sum^N_{n=1}(u,e_n)e_n,\ \  {\bf Q}_Nu:=\sum_{n=N+1}^\infty(u, e_n)e_n.
$$
We are now ready to introduce the main object of study in this paper - an inertial manifold (IM).
\begin{definition}\label{Def2.IM} A set $\Cal M=\Cal M_N$ is an inertial manifold of dimension $N$ for
problem \eqref{3.1}
(with the base $H_N:={\bf P}_N H$) if
\par
1. $\Cal M$ is invariant with respect to the  semigroup $S(t)$: $S(t)\Cal M=\Cal M$.
\par
2. $\Cal M$ is a graph of a Lipschitz continuous function $M:H_N\to{\bf Q}_NH$:
$$
\Cal M=\{p+M(p),\ \ p\in H_N\}.
$$
\par
3. $\Cal M$ possesses an exponential tracking property, namely, for every tra\-jec\-to\-ry $u(t)$ of \eqref{3.1}
 there exists a trace solution $\bar u(t)\in\Cal M$ such that
\begin{equation}\label{2.phase}
 \|u(t)-\bar u(t)\|\le Ce^{-\theta t},\ \ t\ge0
\end{equation}
 for some $\theta>\lambda_N$ and  constant $C=C_u$ which depends on $u$.
\end{definition}
Note that, although only Lipschitz continuity is traditionally required in the definition, usually
 IMs are $C^{1,\eb}$-smooth for some $\eb>0$ (see the discussion below) and are normally hyperbolic. Then the
  exponential tracking property (=existence of an asymptotic phase) as well as robustness with respect
  to perturbations are the standard corollaries of this normal hyperbolicity, see \cite{bates,Fen72,katok,RT96} for the details.
\par
Note also the dynamics of \eqref{3.1} restricted to IM $\Cal M$ is governed by the system of ODEs:
\begin{equation}\label{3.IF}
\frac d{dt}u_N+Au_N={\bf P}_N F(u_N+M(u_N)),\ \ u_N:={\bf P}_Nu\in\R^N
\end{equation}
which is called an inertial form (IF) associated with equation \eqref{3.1}. In the case where the spectral
subspace $H_N$ is used as a base for IM (like in Definition \ref{Def2.IM}), the regularity of the
corresponding vector field in the IF is determined by the regularity of the IM only.
\par
The following theorem is the key result in the theory of IMs.
\begin{theorem}\label{Th2.IM} Let the function $F$ in equation \eqref{3.1} be globally Lipschitz
continuous with Lipschitz constant $L$ and let, for some $N\in\Bbb N$,
the following spectral gap condition
\begin{equation}\label{2.sg}
\lambda_{N+1}-\lambda_N>2L
\end{equation}
be satisfied. Then equation \eqref{3.1} possesses an IM $\Cal M_N$ of dimension $N$.
\end{theorem}
\begin{proof} Although this statement is classical, see e.g. \cite{M91,R94,Z14}, the elements of its proof will be crucially used in what follows, so we sketch them below.
 \par
 To construct the IM, we will use the so-called Perron method, namely, we will prove that,
 for every $p\in H_N$ the problem
 \begin{equation}\label{2.back}
 \Dt u+Au=F(u), \ \ t\le0,\ \ {\bf P}_N u\big|_{t=0}=p
 \end{equation}
 possesses a unique backward solution $u(t)=V(p,t)$, $t\le0$, belonging to the proper weighted space,
  and then define the desired map $M: H_N\to{\bf Q}_NH$ via
  \begin{equation}\label{2.t0}
M(p):={\bf Q}_NV(p,0).
  \end{equation}
 To solve \eqref{2.back} we use the Banach contraction theorem treating the nonlinearity
  $F$ as a perturbation. To this end we need the following two lemmas.
\begin{lemma}\label{Lem2.main} Let $\theta\in(\lambda_N,\lambda_{N+1})$ and let us consider the equation
\begin{equation}\label{2.lin}
\Dt v+Av=h(t),\ \ t\in\R,\ \ h\in L^2_{e^{\theta t}}(\R,H),
\end{equation}
where the space $L^2_{e^{\theta t}}(\R,H)$ is defined via the weighted norm
\begin{equation}\label{2.wl2}
\|h\|_{L^2_{e^{\theta t}}(\R,H)}^2:=\int_{t\in\R}e^{2\theta t}\|h(t)\|^2\,dt<\infty.
\end{equation}
Then, problem \eqref{2.lin} possesses a unique solution $u\in L^2_{e^{\theta t}}(\R,H)$ and the solution
operator $\Cal T:L^2_{e^{\theta t}}\to L^2_{e^{\theta t}}$, $u:=\Cal T h$ satisfies:
\begin{equation}\label{2.sharp}
\|\Cal T\|_{\Cal L(L^2_{e^{\theta t}},L^2_{e^{\theta t}})}=
\frac1{\min\{\theta-\lambda_N,\lambda_{N+1}-\theta\}}.
\end{equation}
\end{lemma}
The proof of this identity is just a straightforward calculation based on decomposition of the solution
$u(t)$ with respect to the base $\{e_n\}_{n=1}^\infty$ and solving the corresponding ODEs, see \cite{Z14}.
\par
The second lemma gives the analogue of this formula for the linear equation on negative semi-axis.

\begin{lemma}\label{Lem2.half} Let $\theta\in(\lambda_N,\lambda_{N+1})$. Then, for any $p\in H_N$
and any $h\in L^2_{e^{\theta t}}(\R_-,H)$, the problem
\begin{equation}
\Dt v+Av=h(t),\ \ t\le0,\ \ {\bf P}_Nv\big|_{t=0}=p
\end{equation}
possesses a unique solution $v\in L^2_{e^{\theta t}}(\R_-,H)$. This solution can be written in the form
$$
v=\Cal T h+\Cal Hp,
$$
where $\Cal T$ is exactly the solution operator constructed in Lemma \ref{Lem2.main} applied to the extension of
the function $h(t)$ by zero for $t\ge0$ and $\Cal H: H_N\to L^2_{e^{\theta t}}(\R_-,H)$ is a
solution operator for the problem with zero right-hand side:
$$
\Cal H(p,t):=\sum_{n=1}^N(p,e_n)e^{-\lambda_nt}.
$$
\end{lemma}
Indeed, this lemma is an easy corollary of Lemma \ref{Lem2.main}, see \cite{Z14}.
\par
We are now ready to prove the theorem. To this end, we fix an optimal value
 $\theta=\frac{\lambda_{N+1}+\lambda_N}2$ and write the equation \eqref{2.back} as a fixed point problem
 \begin{equation}\label{2.fix}
 u=\Cal T\circ F(u)+\Cal H(p)
 \end{equation}
 in the space $L^2_{e^{\theta t}}(\R_-,H)$. Since the norm of the operator $\Cal T$ is equal to
 $\frac2{\lambda_{N+1}-\lambda_N}$ and the Lipschitz constant of $F$ is $L$, the spectral gap
 condition \eqref{2.sg} guarantees that the right-hand side of \eqref{2.fix} is contraction
  for every $p\in H_N$. Thus, by the Banach contraction theorem, for every $p\in H_N$, there exists a
  unique solution $u(t)=V(p,t)$ of problem \eqref{2.back} belonging to $L^2_{e^{\theta t}}(\R_-,H)$ and
   the map $p\to V(p,\cdot)$ is Lipschitz continuous. Due to the parabolic smoothing property,
   we know that
$$
\|u(0)\|\le C(1+\|u\|_{L^2([-1,0],H)})  \text{ and }
\|u(0)-w(0)\|\le C\|u\|_{L^2([-1,0],H)}
$$
for any two backward solutions $u,w$ of \eqref{3.1}, see e.g. \cite{Z14}. In particular, these formulas
 show that the solution $V(p,t)$ is continuous in time ($V(p,\cdot)\in C_{e^{\theta t}}(\R_-,H)$, where
  the weighted space of continuous functions is defined analogously to \eqref{2.wl2}) and the map
  $p\to V(p,\cdot)$ is Lipschitz continuous as a map from $H_N$ to $C_{e^{\theta t}}(\R_-,H)$.
 Thus, formula \eqref{2.t0},
 defines indeed a Lipschitz manifold of dimension $N$ over the base $H_N$ as graph of
 Lipschitz continuous function $M:H_N\to{\bf Q}_NH$.
 \par
 The invariance of this manifold follows by the construction, so we only need to
 verify the exponential tracking property.
\par
Let $u(t)=S(t)u_0$ be an arbitrary solution of problem \eqref{3.1} and let $\phi(t)\in C^\infty(R)$ be a cut-off function
such that $\phi(t)\equiv0$ for $t\le0$ and $\phi(t)\equiv1$ for $t\ge1$. Then the function $\phi(t)u(t)$ is
defined for all $t\in\R$. We seek for the desired solution $\bar u(t)\in\Cal M$ (by the construction of $\Cal M$
such solutions are defined for all $t\in\R$) in the form
\begin{equation}\label{2.trick}
\bar u(t)=\phi(t)u(t)+v(t).
\end{equation}
Inserting this anzatz to \eqref{3.1}, we end up with the equation for $v(t)$:
\begin{equation}\label{2.v}
\Dt v+Av=F(\phi u+v)-\phi F(u)-\phi' u.
\end{equation}
Let $v\in L^2_{e^{\theta t}}(\R,H)$ be a solution of this equation. Then, since $\bar u=v$ for $t\le0$, we
necessarily have $\bar u\in\Cal M$ by the construction of the IM. On the other hand, for $t\ge1$, we have
$v=\bar u-u\in L^2_{e^{\theta t}}([1,\infty),H)$ and using the parabolic smoothing again,
we get the desired estimate \eqref{2.phase}. Thus, we only need to find such a solution $v(t)$. To this end,
we invert the linear part of equation \eqref{2.v} to get the fixed point equation
\begin{equation}\label{2.egg}
v=\Cal T(F(\phi u+v)-\phi F(u)-\phi' u).
\end{equation}
It is straightforward to verify using Lemma \ref{Lem2.main} that the right-hand side of \eqref{2.egg} is a contraction
on the space $L^2_{e^{\theta t}}(\R,H)$ if the spectral gap condition holds, see \cite{Z14}. Thus, the Banach contraction
 theorem finishes the proof of exponential tracking.
\end{proof}
\begin{remark} It is well-known that the spectral gap condition \eqref{2.sg} is sharp in the sense that if it
is violated for some $N$ and $L$, one can find a nonlinearity $F$ such that equation \eqref{3.1} does not
possess an IM of dimension $N$ with base $H_N$, see \cite{R94}.
\par
More recent examples show that if this condition is violated for all $N$:
$$
\sup_{N\in\Bbb N}\{\lambda_{N+1}-\lambda_N\}<2L,
$$
one can construct a smooth nonlinearity $F$ such that equation \eqref{3.1} does not possess
 any Lipschitz or
even Log-Lipschitz finite-dimensional manifold (not necessarily invariant)
which contain the global attractor, see \cite{EKZ13,Z14}.
\end{remark}
\begin{remark} Theorem \ref{Th2.IM} guarantees the existence of an IM $\Cal M_N$ for {\it every} $N$
such that the spectral gap condition \eqref{2.sg} is satisfied. Typically, this $N$ is not unique,
instead, we have a whole sequence $\{N_k\}_{k=1}^\infty$ of $N$s satisfying the spectral gap condition.
Therefore, according to the theorem, we will have a sequence of IMs $\{\Cal M_{N_k}\}_{k=1}^\infty$ of
increasing dimensions: $N_1<N_2<N_3<\cdots$. Moreover, from the explicit description of
an IM using backward solutions of \eqref{2.back}, we see that
\begin{equation}\label{2.many}
\Cal M_{N_1}\subset\Cal M_{N_2}\subset\Cal M_{N_3}\subset\cdots
\end{equation}
In this case it can be also proved that $\Cal M_{N_{k-1}}$ is a normally hyperbolic
submanifold of $\Cal M_{N_k}$.
\end{remark}
Let us now discuss the further regularity of the IM $\Cal M$. To this end, we need one more
auxiliary statement.
\begin{proposition}\label{Prop2.var} Let the spectral gap condition \eqref{2.sg} hold and
 let $u(t)\in C(R_-,H)$
be an arbitrary function. Let also the exponent $\theta\in(\lambda_N,\lambda_{N+1})$ satisfy
\begin{equation}\label{2.gap}
\theta_-:=L+\lambda_N<\theta<\lambda_{N+1}-L:=\theta_+.
\end{equation}
Then, for any $h\in L^2_{e^{\theta t}}(\R_-,H)$ and every $p\in H_N$ the corresponding equation
 of variations
\begin{equation}\label{2.var}
\Dt v+Av-F'(u(t))v=h(t),\ \ t\le0,\ \ {\bf P}_Nv\big|_{t=0}=p
\end{equation}
possesses a unique solution $v\in L^2_{e^{\theta t}}(\R_-,H)\cap C_{e^{\theta t}}(\R_-,H)$ and the following estimate holds:
\begin{equation}\label{2.var-est}
\|v\|_{C_{e^{\theta t}}(\R_-,H)}\le C\|v\|_{L^2_{e^{\theta t}}(\R_-,H)}\le C_{L,\theta}\(\|h\|_{L^2_{e^{\theta t}}(\R_-,H)}+\|p\|\),
\end{equation}
where the constant $C_{L,\theta}$ is independent of $u$, $h$ and $p$.
\end{proposition}
Indeed, equation \eqref{2.var} can be solved via the Banach contraction theorem treating the term $F'(u)v$ as a perturbation
analogously to the non-linear case. Inequalities \eqref{2.gap} guarantee that the map $\Cal T F'(u)v$
 is  a contraction on $L^2_{e^{\theta t}}(\R_-,H)$, due to \eqref{2.sharp}.
\begin{corollary}\label{Cor2.1eb} Let the assumptions of Theorem \ref{Th2.IM} hold and let, in addition, the exponent $\eb\in(0,1]$
 be such that
\begin{equation}\label{2.sgeb}
\lambda_{N+1}-(1+\eb)\lambda_N>(2+\eb)L.
\end{equation}
Assume also that $F\in C^{1,\eb}(H,H)$.
Then the associated IM $\Cal M_N$ is $C^{1,\eb}$-smooth, for any $p,\xi\in H_N$,
the derivative $M'(p)\xi$ can be found as the value of the ${\bf Q}_N$ projection of
$V'(t)=V'(p,t)\xi$ at $t=0$, where the function
$V'$ solves the equation of variations:
\begin{equation}\label{2.vvar}
\Dt V'+AV'-F'(u(t))V'=0,\ \ t\le0,\ {\bf P}_NV'\big|_{t=0}=\xi,\ \ u(t):=V(p,t)
\end{equation}
and
$$
\|M'(p_1)-M'(p_2)\|_{\Cal L(H_N,H)}\le C\|p_1-p_2\|^\eb
$$
for some constant $C$ independent of $p_1,p_2\in H_N$.
\end{corollary}
\begin{proof} Let $p_1,p_2\in H_N$ and $u_i(t):=V(p_i,t)$ be the corresponding trajectories belonging to the IM. Let also
$v(t):=u_1(t)-u_2(t)$ and $\xi:=p_1-p_2$. Then $v$ solves
\begin{equation}\label{2.findif}
\Dt v+Av-L_{u_1,u_2}(t)v=0,\ \ t\le0,\ \ {\bf P}_Nv\big|_{t=0}=\xi,
\end{equation}
where $L_{u_1,u_2}(t):=\int_0^1F'(s u_1(t)+(1-s)u_2(t))\,ds$. Since the norm of $L_{u_1,u_2}(t)$
 does not exceed $L$, Proposition \ref{Prop2.var} is applicable to equation \eqref{2.findif}
 and, therefore, for every $\theta$ satisfying \eqref{2.gap}, we have the estimate
\begin{equation}\label{2.b-est}
 \|v\|_{C_{e^{\theta t}}(\R_-,H)}\le C\|v\|_{L^2_{e^{\theta t}}(\R_-,H)}\le C_\theta\|p_1-p_2\|.
\end{equation}
Note also that the function $V'(p,t)\xi$ is well-defined for all $p,\xi\in H_N$ due to Proposition
\ref{Prop2.var} and satisfy the analogue of \eqref{2.b-est}.
 Let $w(t):=v(t)-V'(p_1,t)\xi$ with $\xi:=p_1-p_2$.
 Then, this function solves
\begin{multline}\label{2.rem}
 \Dt w+Aw-F'(u_1)w=\\=F(u_1)-F(u_2)-F'(u_1)v:=h_{u_1,u_2}(t),\ {\bf P}_Nw\big|_{t=0}=0.
\end{multline}
Since $F\in C^{1,\eb}(H,H)$, by the Taylor theorem, we have
$$
\|h_{u_1,u_2}(t)\|\le C\|v(t)\|^{1+\eb}
$$
which, due to \eqref{2.b-est}, gives
$$
\|h_{u_1,u_2}\|_{L^2_{e^{(1+\eb)\theta t}}(\R_-,H)}\le C\|v\|_{L^2_{e^{\theta t}}(\R_-,H)}
\|v\|^\eb_{C_{e^{\theta t}}(\R_-,H)}\le
 C'\|\xi\|^{1+\eb}.
$$
Fixing now $\theta$ in such a way that $\theta>\theta_-$ and $(1+\eb)\theta<\theta_+$
 (this is possible to do due to assumption \eqref{2.sgeb}) and applying Proposition \ref{Prop2.var}
 to equation \eqref{2.rem}, we finally arrive at
 $$
 \|M(p_2)-M(p_1)-M'(p_1)\xi\|=\|w(0)\|\le C_1\|w\|_{L^2_{e^{(1+\eb)\theta t}}(\R_-,H)}\le C_2\|\xi\|^{1+\eb}
 $$
 and the converse Taylor theorem finishes the proof of the corollary.
 \end{proof}
 The next corollary claims that the constructed manifold $\Cal M$ is actually lives in higher regular space
 $H^2:=D(A)$.
\begin{corollary}\label{Cor2.L2} Let the assumptions of Corollary \ref{Cor2.1eb} hold. Then the manifold $\Cal M$ is simultaneously a
 $C^{1,\eb}$-smooth IM for equation \eqref{3.1} in the phase space $H^2=D(A)$.
\end{corollary}
\begin{proof} This is an almost immediate corollary of the parabolic smoothing property. Indeed,
let us first check that $\Cal M\in H^2$. To this end, it is enough to check that the backward
solution \eqref{2.back} actually belongs to $C_{e^{\theta t}}(\R_-,H^2)$. First, using the $L^2(H^2)$-maximal
 regularity for the solutions of a linear parabolic equation
\begin{equation}\label{2.parlin}
 \Dt v+Av=h(t),\ \ t\le0,
\end{equation}
 namely, that
\begin{multline}\label{2.ref2h2}
\|v\|_{C^\alpha(-1,0;H)}+ \|\Dt v\|_{L^2(-1,0;H)}+\|Av\|_{L^2(-1,0;H)}\le\\\le
  C_\alpha\(\|h\|_{L^2(-2,0,H)}+\|v\|_{L^2(-2,0,H)}\),
\end{multline}
 where $\alpha\in(0,\frac12)$,
we end up with the estimate
\begin{multline}\label{2.uh2}
\|u\|_{C^\alpha(-1,0;H)}\le C_\alpha\(\|F(u)\|_{L^2(-2,0;H)}+\|u\|_{L^2(-2,0;H)}\)\le\\\le
 C_{\alpha,\theta}(1+\|u\|_{L^2_{e^{\theta t}}(\R_-,H)})\le
  C(1+\|p\|),
\end{multline}
where $\alpha\in(0,\frac12)$. Second, using the $C^\alpha(H)$-maximal regularity for solutions
 of \eqref{2.parlin} and the obvious estimate
 $$
 \|F(u)\|_{C^\alpha(-2,0;H)}\le \|F\|_{C^\alpha(H,H)}(1+\|u\|_{C^\alpha(-2,0;H)}^\alpha),
 $$
 we arrive at
 \begin{multline}
\|\Dt u\|_{C^\alpha(-1,0;H)}+\|Au\|_{C^\alpha(-1,0;H)}\le\\\le
C\(\|F(u)\|_{C^\alpha(-2,0;H)}+\|u\|_{C^\alpha(-2,0;H)}\)\le\\\le C_1\(1+\|u\|_{C^\alpha(-2,0;H)}\)\le
C_2\(1+\|p\|\)
 \end{multline}
 and the fact that $M(p)$ belongs to $H^2$ is proved. The fact that $M$ is $C^{1,\eb}$-smooth as a map from $H_N$ to $H^2$ can be
  verified analogously and the corollary is proved.
\end{proof}
\begin{remark} The analogue of Corollary \ref{Cor2.1eb} holds for higher derivatives as well.
For instance, if we want to have $C^{n,\eb}$-smooth IM, we need to require that
\begin{equation}\label{2.sg-n}
\lambda_{N+1}-(n+\eb)\lambda_N>(n+1+\eb)L.
\end{equation}
To verify this, we just need to define the higher order Taylor jets for the IM $\Cal M$
using second, third, etc.,
equations of variations for \eqref{2.back} and use again Proposition \ref{Prop2.var}. For instance,
the second derivative $V''=V''(p,t)[\xi,\xi]$ solves
\begin{multline}\label{2.s-der}
\Dt V''+AV''-F'(u(t))V''=F''(u(t))[V'(p,t)\xi,V'(p,t)\xi],\\ {\bf P}_N V''\big|_{t=0}=0,\  u(t):=V(p,t).
\end{multline}
According to Proposition \ref{Prop2.var}, in order to be able to solve this equation, we need
 $\theta_+>2\theta_-$ (since $V'\in L^2_{e^{\theta t}}$ with $\theta>\theta_-$ and
the right-hand side $F''(u)[V',V']\in L^2_{e^{2\theta t}}$) which gives \eqref{2.sg-n} for $n=2$.
\par
We believe that sufficient condition \eqref{2.sg-n} for the existence of $C^{n,\eb}$-smooth IM
 is sharp for any $n$ and $\eb$, but we restrict ourselves by recalling below the classical
  counterexample of G. Sell to the existence of $C^2$-smooth IM which demonstrates the sharpness of \eqref{2.sg-n} for $n=2$, see \cite{S-sell}.
\end{remark}
\begin{example}\label{Ex2.sell} Let $H:=l^2$ (space of square summable sequences with the standard inner product) and let
 us consider the following particular case of equation \eqref{3.1}:
 \begin{equation}\label{2.counter}
 \frac d{dt}u_1+u_1=0,\ \ \frac d{dt}u_n+2^{n-1}u_n=u_{n-1}^2, \ n=2,3,\cdots
 \end{equation}
Here $\lambda_n=2^{n-1}$ and we have a set of resonances $2\lambda_n=\lambda_{n+1}$ which prevent the
existence of any finite-dimensional invariant local manifold of dimension greater than zero which is $C^2$-smooth
 and contains zero. Note that the non-linearity here is locally smooth near zero and since we are
  interested in {\it local} invariant manifolds near zero, the behaviour of it outside the small
  neighbourhood of zero is not important (we may always cut-off it outside of the neighbourhood to get
   global Lipschitz continuity). Moreover, since $F'(0)=0$, decreasing the size of the neighbourhood we
    may make the Lipschitz constant $L$ as small as we want. Thus, according to Corollary \ref{Cor2.1eb},
    for any $N\in\Bbb N$, there exists a local invariant manifold $\Cal M_N$ of dimension $N$ with
    the base $H_N$ which is $C^{1,\eb}$-smooth for any $\eb<1$.
    \par
    Let us check that $C^2$-smooth invariant local manifold does not exist. Indeed, let $\Cal M_N$
    be such a manifold of dimension $N$. Then, since the tangent plane $T\Cal M_N(0)$ to this
    manifold at zero
    is invariant with respect to $A$ (due to the fact that $F'(0)=0$), we must have
    $$
    H'_N:=T\Cal M_N(0)=\spann\{e_{n_1},\cdots,e_{n_N}\}
    $$
    for some $n_1<n_2<\cdots<n_N$. Thus, the manifold $\Cal M_N$ can be presented locally near zero as
     a graph of $C^2$-function $M:H_N'\to (H_N')^\perp$ such that $M(0)=M'(0)=0$. In particular, expanding $M$ in
     Taylor series near zero, we have
$$
u_{n_N+1}=(M(u_{n_1},\cdots,u_{n_N}),e_{n_N+1})=cu_{n_N}^2+\cdots
$$
Let us try to compute the constant $c$. Inserting this  in the $(n_N+1)$-th equation and using the
 invariance, we get
\begin{multline}
\Dt u_{n_N+1}+2^{n_N}u_{n_N+1}=2c\Dt u_{n_N}u_{n_N}+2^{n_N}cu_{n_N}^2+\cdots=\\=-2c2^{n_N-1}u_{n_N}^2+2^{n_N}cu_{n_N}^2+\cdots=
0+\cdots=u_{n_N}^2
\end{multline}
which gives $0=1$. Thus, the manifold $\Cal M_N$ cannot be $C^2$-smooth.
\end{example}

\begin{remark} Note that in the case where $A$ is an elliptic operator of order $2k$ in a bounded
 domain $\Omega$ of $\R^d$, we have $\lambda_n\sim Cn^{2k/d}$ due to the Weyl asymptotic. Thus, one may expect
  in general only the gaps of the size
\begin{equation}\label{2.weyl}
 \lambda_{N+1}-\lambda_N\sim CN^{\frac{2k}d-1}\sim C'\lambda_N^{1-\frac d{2k}}
\end{equation}
  which is much weaker than \eqref{2.sg-n} with $n>1$. Sometimes the exponent in the
  right-hand side of \eqref{2.weyl} may be improved due to big multiplicity of eigenvalues (e.g. for
  the Laplace-Beltrami operator on a sphere $S^d$, we have $\lambda_N^{1/2}$ there for all $d$), but this
  exponent is always {\it less than one} in all more or less realistic examples. Thus, the existence of 
  $C^n$-smooth IMs with $n>1$ looks not realistic and could be obtained in general only for
   bifurcation problems where, e.g. $\lambda_1,\cdots,\lambda_N$ are close to zero, $\lambda_{N+1}$ is of order
    one and $L$ is small.
\par
In contrast to this, if the spectral gap conditions \eqref{2.sg} are satisfied for some $N$, we always can
 find a small positive $\eb=\eb_N$ such that \eqref{2.sgeb} will be also satisfied. Thus, if the nonlinearity
 $F$ is smooth enough, we automatically get a $C^{1,\eb}$-smooth IM for some small $\eb$ depending on $N$ and $L$.
\end{remark}
\begin{remark}\label{Rem2.strange} Let $\bar u(t)$ be a trajectory of \eqref{3.1} belonging to the IM,
i.e.
$$
{\bf Q}_N\bar u(t)\equiv M_N({\bf P}_N\bar u (t))
$$
and let $\bar u_N:={\bf P}_N\bar u(t)$. Then, we may write a linearization near the trajectory
 $\bar u(t)$ in two natural ways. First, we may just linearize equation \eqref{3.1} without using
  the fact that $\bar u\in\Cal M_N$. This gives the equation
  \begin{equation}\label{2.lin-inf}
\Dt v+Av-F'(\bar u)v=h(t)
  \end{equation}
  which we have used above to get the existence of the IM, its smoothness and exponential tracking.
  \par
  Alternatively, we may linearize the reduced ODEs \eqref{3.IF}:
  \begin{equation}\label{2.lin-fin}
\Dt v_N+Av_N-F'(\bar u)(v_N+M'_N(\bar u)v_N)=h_N(t).
  \end{equation}
Of course, these two equations are closely related. Namely, if $v_N(t)$ solves \eqref{2.lin-fin},
then the function
\begin{equation}\label{2.lift}
v(t):=v_N(t)+M'_N(\bar u(t))v_N(t)
\end{equation}
solves \eqref{2.lin-inf} with
\begin{equation}\label{2.lift-h}
h(t):=h_N(t)+M'_N(\bar u(t))h_N(t).
\end{equation}
Vice versa, if $h(t)$ satisfies \eqref{2.lift-h} and the solution $v(t)$ of \eqref{2.lin-inf} satisfies
\eqref{2.lift} for some $t$, then it satisfies \eqref{2.lift} for all $t$ and $v_N(t):={\bf P}_N v(t)$
solves \eqref{2.lin-fin}.
\par
This equivalence is a straightforward corollary of the invariance of the manifold $\Cal M_N$ and we leave
 its rigorous proof to the reader.
\end{remark}

\section{Main result}\label{s3}
In this section we develop an alternative approach for constructing $C^n$-smooth IFs which does not require
huge spectral gaps. The key idea is to require instead the existence of {\it many} spectral gaps and to use
 the second spectral gap in order to solve equation \eqref{2.s-der} for the second derivative, the third gap
  to solve the appropriate equation for the third derivative, etc. Of course, this will not allow us to
  construct $C^n$-smooth IM (we know that it may not exist for $n>1$, see Example \ref{Ex2.sell}).
  Instead, for every $p\in\Cal M_{N_2}$ and the corresponding trajectory $u=V(p,t)$,
  we construct the corresponding Taylor jet $J_\xi^n V(p,t)$ of length $n+1$ belonging to the
   space $\Cal P^n(H_{N_n},H)$ for all $t\le0$, where $N_k$ is the dimension of the IM $\Cal M_{N_k}$ built
    up on the $k$th spectral gap. These jets must be constructed in such a way that the
    compatibility conditions are satisfied. Then, the Whitney embedding theorem will give
     us the desired smooth extension of the initial IM. To be more precise, we give the
      following definition of such a smooth extension.

\begin{definition}\label{Def3.ext-IF} Let  equation \eqref{3.1} possess at least two spectral gaps which
 corresponds to the dimensions $K_1$ and $K_2$ and let $\eb>0$ be a small number. Denote the corresponding
  IMs by  $\Cal M_{K_1}$ and $\Cal M_{K_2}$ respectively, the corresponding $C^{1,\eb}$-functions
  generating these manifolds are denoted by $M_{K_1}$ and $M_{K_2}$ respectively. A
  $C^{n,\eb}$-smooth submanifold $\tilde{\Cal M}_{K_2}$
  (not necessarily invariant) of dimension $K_2$ is called a $C^n$-extension of the IM $\Cal M_{K_1}$ if
  the following conditions hold:
  \par
  1) $\tilde{\Cal M}_{K_2}$ is a graph of a $C^{n,\eb}$-smooth function $\tilde M_{K_2}:{\bf P}_{K_2}H\to
   {\bf Q}_{K_2}H$.
  \par
  2) $\tilde M_{K_2}\big|_{{\bf P}_{K_2}\Cal M_{K_1}}={\bf Q}_{K_2}M_{K_1}$ and therefore
  $\Cal M_{K_1}\subset \tilde {\Cal M}_{K_2}$.
\par
3) $\tilde M_{K_2}$ is $\mu$-close in the $C^1_b$-norm to $M_{K_2}$
 for a sufficiently small $\mu$.
\end{definition}
\begin{remark}\label{Rem3.main}
The $C^{n,\eb}$ dynamics on the extended IM $\tilde{\Cal M}_{K_2}$ is naturally defined via
\begin{equation}\label{3.IF-ext}
\Dt u_{K_2}+A u_{K_2}={\bf P}_{K_2} F(u_{K_2}+\tilde M_{K_2}(u_{K_2})),\ \ u_{K_2}\in H_{K_2}
\end{equation}
and $u(t):=u_{K_2}(t)+\tilde M_{K_2}(u_{K_2}(t))$. Obviously, the manifold $\tilde{\Cal M}_{K_2}$ is
 invariant with respect to the dynamical system thus defined. Moreover, due to the
 second condition of Definition \ref{Def3.ext-IF}, the $C^{1,\eb}$-submanifold
 ${\bf P}_{K_2}\Cal M_{K_1}\subset H_{K_2}$ is invariant with respect to equation \eqref{3.IF-ext}
 and the restriction of \eqref{3.IF-ext} coincides with the initial IF \eqref{3.IF} generated
  by the IM $\Cal M_{K_1}$. Thus, system of ODEs \eqref{3.IF-ext} is indeed a
  smooth extension of the IF \eqref{3.IF}.
  \par
  Finally, the 3rd condition of Definition \ref{Def3.ext-IF} guarantees that ${\bf P}_{K_2}\Cal M_{K_1}$
   is a normally hyperbolic stable invariant manifold for \eqref{3.IF-ext}
    (since it is so for the IF generated by the function $M_{K_2}$). This means that
    ${\bf P}_{K_2}\Cal M_{K_1}$
    also  possesses an exponential tracking property. Thus, the limit dynamics generated by the extended IF
     coincides with the one generated by the initial abstract parabolic equation \eqref{3.1}.
\end{remark}
We are now ready to state the main result of the paper.
\begin{theorem}\label{Th3.main} Let the nonlinearity $F:H\to H$ in equation \eqref{3.1} be smooth
and all its derivatives
be globally bounded. Let also   the
 following form of spectral gap conditions be satisfied:
 \begin{equation}\label{3.sginf}
 \limsup_{N\to\infty}(\lambda_{N+1}-\lambda_N)=\infty.
 \end{equation}
 Then, for any $n\in\Bbb N$ and any $\mu>0$, equation \eqref{3.1} possesses a $C^{n,\eb}$-smooth
 extension $\tilde{\Cal M}_{N_n}$ of the initial IM $\Cal M_{N_1}$ (where $N_1$ is the first $N$ which satisfies
 the spectral gap condition \eqref{2.sg} and $\eb>0$ is small enough) such that $\tilde{\Cal M}_{N_n}$ is $\mu$-close to
  the IM $\Cal M_{N_n}$ in the $C^{1}_b$-norm.
\end{theorem}
\begin{proof}[Proof for $n=2$] Let $N_1$ be the first $N$ for which the spectral gap condition
 \eqref{2.sg} is satisfied with $L:=\|F'\|_{C_b(H,\Cal L(H,H))}$ and let the
 corresponding $\Cal M_1$ be $C^{1,\eb}$-smooth IM which exists due to Theorem \ref{Th2.IM} and
 Corollary \ref{Cor2.1eb}. Recall that for any $p\in H$, we have a solution $V(p,t)$  of
 problem \eqref{2.back} (where $p$ is replaced by ${\bf  P}_{N_1}p$) and its Frechet derivative
 $V'_\xi(t):=V'(p,t)\xi$ in $p$ satisfies equation of variations
\eqref{2.vvar} and belongs to the space $L^2_{e^{\theta_1 t}}(\R_-,H)$ for any $\theta_1$
satisfying \eqref{2.gap}. Moreover, for any other $p_1\in H$, we have the estimate
\begin{equation}\label{3.c1}
\|V(p_1,t)-V(p,t)-V'_\xi(t)\|_{L^2_{e^{\theta_1(1+\eb) t}}(\R-,H)}\le C\|{\bf P}_{N_1}(p-p_1)\|^{1+\eb},
\end{equation}
where $\eb>0$, $\xi:=p_1-p$ and $C$ is independent of $p$ and $p_1$.
\par
Let now $N_2>N_1$ be the first $N$ which satisfies
\begin{equation}\label{3.sg2}
\lambda_{N_2+1}-\lambda_{N_2}-\lambda_{N_1}>3L
\end{equation}
(such $N$ exists due to condition \eqref{3.sginf}). Then, we have the
corresponding $C^{1,\eb}$-smooth IM $\Cal M_{N_2}$. Let us denote by $W(p,t)$, $p\in H$,
the corresponding solution of \eqref{2.back} (where $N$ is replaced by $N_2$ and $p$ is replaced by ${\bf P}_{N_2}p$). This solution belongs
to $L^2_{e^{\theta_2 t}}(\R_-,H)$ with $\theta_2$ satisfying \eqref{2.gap} (with $N$ replaced by $N_2$).
Moreover, analogously to \eqref{3.c1}, we have
\begin{equation}\label{3.c12}
\|W(p_1,t)-W(p,t)-W'_\xi(t)\|_{L^2_{e^{\theta_2(1+\eb) t}}(\R_-,H)}\le C\|{\bf P}_{N_2}(p-p_1)\|^{1+\eb},
\end{equation}
where $W'_\xi(t)=W'(p,t)\xi$ solves \eqref{2.vvar} with $N$ replaced by $N_2$.
We also know that $V(p,t)=W(p,t)$ if $p\in\Cal M_{N_1}$ and, therefore, due to \eqref{3.c1} and \eqref{3.c12},
\begin{multline}\label{3.dif-der}
\|V'(p,\cdot)\xi-W'(p,\cdot)\xi\|_{L^2_{e^{\theta_2(1+\eb) t}}(\R_-,H)}\le\\\le
 C\|{\bf P}_{N_2}\xi\|^{1+\eb},\ \xi=p_1-p,\ \ p,p_1\in\Cal M_{N_1}.
\end{multline}
Let us define for every $p\in \Cal M_{N_1}$ and every $\xi\in H$ the "second derivative"
$W''_\xi=W''(p,t)[\xi,\xi]$
of the trajectory $u(t)=W(p,t)=V(p,t)$ as a solution of the following problem
\begin{multline}\label{3.2der}
\Dt W''_\xi+AW''_\xi-F'(V(p,t))W''_\xi=\\=2F''(V(p,t))[V'_\xi,W'_\xi]-F''(V(p,t))[V'_\xi,V'_\xi],\ \
{\bf P}_{N_2}W''_\xi\big|_{t=0}=0.
\end{multline}
Note that the right-hand side of this equation belongs to the weighted space  $L^2_{e^{(\theta_1+\theta_2)t}}(\R_-,H)$,
where the exponents $\theta_1$ and $\theta_2$ satisfy assumption \eqref{2.gap} with
 $N=N_1$ and $N=N_2$ respectively. Moreover, due to assumption \eqref{3.sg2}, it is possible to
 fix $\theta_1$ and $\theta_2$ in such a way that the exponent $\theta_1+\theta_2$ still satisfies
  \eqref{2.sg} with $N=N_2$. Thus, by Proposition \ref{Prop2.var}, there exists a unique solution of
   \eqref{3.2der} belonging to the space $L^2_{e^{(\theta_1+\theta_2)t}}(\R_-,H)$ and the function
    $W''_\xi$ is well-defined and satisfies
$$
\|W''_\xi\|_{C_{e^{(\theta_1+\theta_2)t}}(\R_-,H)}\le
C\|W''_\xi\|_{L^2_{e^{(\theta_1+\theta_2)t}}(\R_-,H)}\le C^2\|\xi\|^2,
$$
where $C$ is independent of $p$.
\par
Let us define the desired quadratic polynomial $\xi\to J_\xi^2 W(p,t)$, $p\in\Cal M_{N_1}$ as follows:
\begin{equation}\label{3.2-jet}
J_\xi^2 W(p,t):=V(p,t)+W'(p,t)\xi+\frac12W''(p,t)[\xi,\xi],\ \ \xi\in H.
\end{equation}
We need to verify the compatibility conditions for these "Taylor jets" on $p\in\Cal M_{N_1}$. It
is straightforward to check using $F\in C^{2,\eb}$, $V,W\in C^{1,\eb}$ and Proposition
\ref{Prop2.var} that
$$
\|W''(p_1,\cdot)[\xi,\xi]-W''(p,\cdot)[\xi,\xi]\|_{L^2_{e^{(\theta_1+\theta_2+\eb)t}}(\R_-,H)}
\le C\|\xi\|^2\|p-p_1\|^\eb
$$
for $p,p_1\in\Cal M_{N_1}$. This gives us the desired compatibility condition for
the second derivative, see \eqref{2.13} for $n=l=2$.
\par
Let us now verify the compatibility conditions for the first derivative ($l=1$, $n=2$ in \eqref{2.13}).
To this end, we need to expand  the difference $w(t):=W'(p_1,t)\xi-W'(p,t)\xi$, $p,p_1\in\Cal M_{N_1}$ in
 terms of $\delta=p-p_1$. By the definition of $W'$, this function satisfies the equation
 \begin{multline}\label{3.2der-bad}
\Dt w+Aw-F'(V(p,t))w=(F'(V(p_1,t))-F'(V(p,t)))W'(p_1,t)\xi=\\=F''(V(p,t))[V'(p,t)\delta,W'(p,t)\xi]+h(t),\ \ {\bf P}_{N_2}w\big|_{t=0}=0,
\end{multline}
 where the reminder $h$ satisfies
 $$
 \|h\|_{L^2_{e^{(\theta_1+\theta_2+\eb)t}}(\R_-,H)}\le C\|\delta\|^{1+\eb}\|\xi\|
 $$
 for sufficiently small positive $\eb$ (this also follows from the fact that
 $F$ is smooth and $V,W\in C^{1,\eb}$). Thus, the reminder $h$ in the right-hand side of \eqref{3.2der-bad} is of
 higher order in $\delta$ and, by this reason, is not essential, so we need to study the bilinear form (w.r.t. $\delta,\xi$)
  in the right-hand side. Note that, in contrast to the case where the IM is $C^2$, this form is even not symmetric,
  so it should be corrected. Namely, we write the identity
\begin{multline}
F''(V(p,t))[V'(p,t)\delta,W'(p,t)\xi]=\\=
\left\{F''(V(p,t))[V'(p,t)\delta,W'(p,t)\xi]+F''(V(p,t))[V'(p,t)\xi,W'(p,t)\delta]-\right.\\-\left.
F''(V(p,t))[V'(p,t)\delta,V'(p,t)\xi]\right\}-\\-
F''(V(p,t))[V'(p,t)\xi,W'(p,t)\delta-V'(p,t)\delta]
\end{multline}
and note that the first term in the right-hand side is nothing more than the symmetric bilinear form
 which corresponds to the quadratic form
 $$
 2F''(V(p,t))[V'(p,t)\xi,W'(p,t)\xi]-F''(V(p,t))[V'(p,t)\xi,V'(p,t)\xi]
 $$
 used in \eqref{3.2der} to define $W''$ and the second term is of order $\|\delta\|^{1+\eb}\|\xi\|$ due
 to estimate \eqref{3.dif-der} (where $\xi$ is replaced by $\delta$) and the growth rate of this
  term does not exceed $e^{-(\theta_1+\theta_2+\eb)t}$ as $t\to-\infty$. Thus,
  by Proposition \ref{Prop2.var}, we have
$$
\|w-W''(p,\cdot)[\delta,\xi]\|_{L^2_{e^{(\theta_1+\theta_2+\eb)t}}(\R_-,H)}\le C\|\delta\|^{1+\eb}\|\xi\|
$$
and the compatibility condition for $l=1$ is verified.
\par
Finally, let us check the zero order compatibility condition ($l=0$, $n=2$ in \eqref{2.13}). Let
$$
R(t):=V(p_1,t)-V(p,t)-W'(p,t)\delta-\frac1{2!}W''(p,t)[\delta,\delta].
$$
Then, as elementary computations show, this function satisfies the equation
\begin{multline}\label{3.T}
\Dt R+A R-F'(V(p,t))R=\\=\left\{F(V(p_1,t))-F(V(p,t))-F'(V(p,t))(V(p_1,t)-V(p,t))\right\}-\\-
\frac1{2!}\(2F''(V(p,t))[V'(p,t)\delta,W'(p,t)\delta]-F''(V(p,t))[V'(p,t)\delta,V'(p,t)\delta]\),\\
{\bf P}_{N_2}\big|_{t=0}R=0.
\end{multline}
Since $F\in C^{2,\eb}$ and $V\in C^{1,\eb}$, the first term in the right-hand side equals to
\begin{equation}\label{3.good}
\frac1{2!}F''(V(p,t))[V'(p,t)\delta,V'(p,t)\delta]
\end{equation}
up to the controllable in $L^2_{e^{(\theta_1+\theta_2+\eb)t}}(\R_-,H)$-norm remainder of
order $\|\delta\|^{2+\eb}$. The second term can be simplified using \eqref{3.dif-der} and
also equals to \eqref{3.good} up to higher order terms. Thus, the right-hand side of \eqref{3.T} vanishes
 up to terms of order $\|\delta\|^{2+\eb}$ and Proposition \ref{Prop2.var} gives us that
\begin{equation}\label{3.0comp}
\|R\|_{L^2_{e^{(\theta_1+\theta_2+\eb)t}}(\R_-,H)}\le C\|\delta\|^{2+\eb}
\end{equation}
 for some positive $\eb$. This finishes the verification of the compatibility conditions.
 \par
 We are now ready to use Whitney extension theorem. To this end, we first recall that the IM
 $\Cal M_{N_2}$ is a graph of the $C^{1,\eb}$-function $M_{N_2}:{\bf P}_{N_2}H\to {\bf Q}_{N_2}H$
 which is defined via $M_{N_2}(p):={\bf Q}_{N_2}W(p,0)$, $p\in {\bf P}_{N_2}H=H_{N_2}$ (all functions $V,W,W',W''$
 defined above depend only on ${\bf P}_{N_2}$-component of $p\in H$, so without loss of generality
  we may assume that $p,\xi,\delta\in H_{N_2}$ (we took them from $H$ in order to simplify
   the notations only). Thus, projecting the constructed Taylor jets to $t=0$ and ${\bf Q}_{N_2}H$, we get
   the $C^{1,\eb}$-function $M_{N_2}(p)$ restricted to the invariant set $p\in {\bf P}_{N_2}\Cal M_{N_1}$
    and a family of quadratic polynomials
    $$
    J^2_\xi M_{N_2}(p):={\bf Q}_{N_2}J^2_\xi W(p,0)
    $$
    which satisfy the compatibility conditions on $p\in{\bf P}_{N_2}\Cal M_{N_1}$. Therefore, since
     $H_{N_2}$ is finite-dimensional, Whitney extension theorem gives the existence of
     a $C^{2,\eb}$-function $\widehat M_{N_2}:{\bf P}_{N_2}H\to{\bf Q}_{N_2}H$ such that
     $$
     J^2_\xi \widehat M_{N_2}(p)=J^2_\xi M_{N_2}(p),\ \ p\in{\bf P}_{N_2}\Cal M_{N_1}.
     $$
     Thus, the desired $C^{2+\eb}$-extension of the IM $\Cal M_{N_1}$ is "almost" constructed.
     It only remains to take care about the closeness in the $C^1$-norm. To this end, for any small $\nu>0$, we
      introduce a cut-off function $\rho_\nu\in C^\infty(H_{N_2},\R)$ such that $\rho(p)\equiv0$ if
      $p$ belongs to the $\nu$-neighbourhood $\Cal O_\nu$ of ${\bf P}_{N_2}\Cal M_{N_1}$ and
      $\rho(p)\equiv1$ if
  $p\notin \Cal O_{2\nu}$. Moreover, since ${\bf P}_{N_2}\Cal M_{N_1}$ is $C^{1,\eb}$-smooth, we may require
   also that
\begin{equation}\label{3.varphi}
|\nabla_p\rho(p)|\le C\nu^{-1},
\end{equation}
      where the constant $C$ is independent of $\nu$. Finally, we define
      \begin{equation}\label{3.small}
      \tilde M_{N_2}(p):=(1-\rho_\nu(p))\widehat M_{N_2}(p)+\rho_\nu(p) (\Bbb S_{\nu^2} M_{N_2})(p),
      \end{equation}
      where $\Bbb S_\mu$ is a standard mollifying operator:
$$
(\Bbb S_\mu f)(p):=\int_{\R^{N_2}}\beta_\mu(p-q)f(q)\,dq
$$
and the kernel $\beta_\mu(p)=\frac1{\mu^{N_2}}\beta_1(p/\mu)$ and $\beta_1(p)$ is a smooth, non-negative
function with compact support satisfying $\int_{\R^{N_2}}\beta_1(p)\,dp=1$.
\par
       We claim that $\tilde M_{N_2}$ is a
      desired extension. Indeed, $\tilde M_{N_2}(p)\equiv\widehat M_{N_2}(p)$ in $\Cal O_\nu$ and therefore
      $\tilde M_{N_2}$ and $M_{N_2}$ coincide on ${\bf P}_{N_2}\Cal M_{N_1}$. Obviously, $\tilde M_{N_2}$
      is $C^{2,\eb}$-smooth. To verify closeness, we note that
\begin{multline}\label{3.MM}
      \widetilde M_{N_2}(p)-M_{N_2}(p)=(1-\rho_\nu(p))(\widehat M_{N_2}(p)-M_{N_2}(p))+\\+
      \rho_\nu(p)((\Bbb S_{\nu^2}M_{N_2})(p)-M_{N_2}(p))
\end{multline}
      Using the fact that $M_{N_2}\in C^{1,\eb}$ together with the standard estimates for the
       mollifying operator, we get
       $$
       \|(\Bbb S_{\nu^2}M_{N_2})(p)-M_{N_2}(p)\|\le C\nu^2,\ \
       \|\nabla_p(\Bbb S_{\nu^2}M_{N_2})(p)-\nabla_pM_{N_2}(p)\|\le C\nu^{2\eb}
       $$
which together with \eqref{3.varphi} shows that the $C^1$-norm of the
 second term in the right-hand side of \eqref{3.MM} is of order $\nu^{2\eb}$. To estimate the first
  term, we use that both functions $\widehat M_{N_2}(p)$ and $M_{N_2}(p)$ are at least $C^{1,\eb}$-smooth
  and
  $$
  \widehat M_{N_2}(p)=M_{N_2}(p),\ \ \nabla_p\widehat M_{N_2}(p)=\nabla_pM_{N_2}(p),\ \ p\in{\bf P}_{N_2}\Cal M_{N_1}.
  $$
  By this reason,
  $$
  \|\widehat M_{N_2}(p)-M_{N_2}(p)\|\le C\nu^{1+\eb},\ \|\nabla_p\widehat M_{N_2}(p)-\nabla_pM_{N_2}(p)\|\le C\nu^\eb
  $$
  for all $p\in\Cal O_{2\nu}$. Thus, using \eqref{3.varphi} again, we see that
  $$
\|\widetilde M_{N_2}(\cdot)-M_{N_2}(\cdot)\|_{C^1_b(H_{N_2},H)}\le C\nu^\eb.
$$
This finishes the proof of the theorem for the case $n=2$.
  \end{proof}
\begin{proof}[Proof for general $n\in\Bbb N$] We will proceed by induction with respect to $n$.
Assume that for some $n\in\Bbb N$, we have already constructed the $C^{1,\eb}$-smooth inertial
manifold $\Cal M_{N_n}$ which is a graph of a map $M_{N_{n}}:{\bf P}_{N_n}H\to {\bf Q}_{N_n}H$
and this map is constructed via the solution $V(p,t)$, $t\le0$, $p\in H$ of the backward
problem \eqref{2.back} where $N$ is replaced by $N_n$. Recall that this manifold is constructed using the
 $n$th spectral gap. Assume also that, for every $p\in{\bf P}_{N_n}\Cal M_{N_1}$, we have
  already constructed the $n$th Taylor jet $J^n_\xi V(p,t)$ such that the compatibility conditions
  up to order $n$ are satisfied. In contrast to the proof for the case $n=2$, it is convenient for
   us to write these conditions in the form of \eqref{2.12}:
\begin{multline}\label{3.compd}
   \|J_\xi^nV(p_1,\cdot)-
   J^n_{\xi+\delta}V(p,\cdot)\|_{L^2_{e^{(\theta_n+(n-1)\theta_{n-1}+\eb)t}}(\R_-,H)}\le\\\le
   C(\|\delta\|+\|\xi\|)^{n+\eb}.
\end{multline}
Here $\xi\in H$ is arbitrary, $\delta:=p_1-p$, $\eb>0$ and $\theta_1<\theta_2\cdots<\theta_n$ are the exponents
 which satisfy conditions \eqref{2.gap} for $N=N_1,\cdots,N_n$. In order to simplify  notations,
 we will write below
\begin{equation}\label{3.compd1}
J_\xi^nV(p_1)-
   J^n_{\xi+\delta}V(p)=O_{\theta_n+(n-1)\theta_{n-1}+\eb}\((\|\delta\|+\|\xi\|)^{n+\eb}\)
\end{equation}
instead of \eqref{3.compd} and also in similar situations. Rewriting \eqref{3.compd1} in terms of
 truncated jets, we have
\begin{equation}\label{3.compt}
j_\xi^nV(p_1)+j^n_\delta V(p)-
   j^n_{\xi+\delta}V(p)=O_{n\theta_n+\eb}\((\|\delta\|+\|\xi\|)^{n+\eb}\),
\end{equation}
where we have  used that $\theta_{n-1}<\theta_n$. We also need the induction assumption that \eqref{3.compt}
holds for every $m\le n$, namely,
 \begin{equation}\label{3.compd3}
J_\xi^mV(p_1)-
   J^m_{\xi+\delta}V(p)=O_{m\theta_n+\eb}\((\|\delta\|+\|\xi\|)^{m+\eb}\).
\end{equation}
Let us now consider the $(n+1)$th spectral gap at $N=N_{n+1}$ which is the first $N$ satisfying
\begin{equation}\label{4.gapnp1}
\lambda_{N_{n+1}}+L+n(\lambda_{N_n+1}-L)<\lambda_{N_{n+1}+1}-L.
\end{equation}
Let $\Cal M_{N_{n+1}}$ be the corresponding IM which is generated by the backward solution $W(p,t)$ of
 problem \eqref{2.back} with $N$ replaced by $N_{n+1}$.
We need to define the $(n+1)$th Taylor jet $J_\xi^{n+1}W(p,t)$ for the function $W(p,t)$:
\begin{equation}\label{3.Wjet}
J_\xi^{n+1} W(p,t)=W(p,t)+\sum_{k=1}^{n+1}\frac1{k!}W^{(k)}(p,t)[\{\xi\}^k],
\end{equation}
$\xi\in H$ and  $p\in {\bf P}_{N_{n+1}}\Cal M_{N_1}$ and to verify the compatibility
conditions of order $n+1$. Keeping in mind already considered case $n=1$ and $n=2$, we introduce
 the required jet \eqref{3.Wjet} as a backward solution of the following equation:
\begin{multline}\label{3.jeteq}
\Dt J^{n+1}_\xi W(p)+AJ^{n+1}_\xi W(p)\!=\!F^{[n+1]}(p,\xi),\\  {\bf P}_{N_{n+1}}J^{n+1}_\xi(p)\big|_{t=0}={\bf P}_{N_{n+1}}(p+\xi),
\end{multline}
where
\begin{multline}\label{3.Fn}
F^{[n+1]}(p,\xi,t):=F(W(p,t))+F'(W(p,t))j_\xi^{n+1}W(p,t)+\\+\sum_{k=2}^{n+1}
\frac1{k!}\(kF^{(k)}(W(p,t))[\{j_\xi^n V(p,t)\}^{k-1},j_\xi^n W(p,t)]-\right.\\
\left.-(k-1)F^{(k)}(W(p,t))[\{j_\xi^n V(p,t)\}^k]\).
\end{multline}
Symbol $"[n+1]"$ means that we have dropped out all terms of order greater than $n+1$ from the
right-hand side, so $F^{[n+1]}$ is a polynomial of order $n+1$ in $\xi\in H$. Alternatively,
the dropping out procedure means that we replace
\begin{equation}\label{3.cor}
\{j_\xi^n V(p)\}^k\rightarrow
 \sum_{\substack{n_1+\cdots+n_k\le n+1\\ n_i\in\Bbb N}}B_{n_1,\cdots,n_k}
 \{j_\xi^{n_1} V(p),\cdots,j_\xi^{n_k} V(p)\},
\end{equation}
where the numbers $B_{n_1,\cdots,n_k}\in\R$ are chosen in such a way
that polynomials in the left and right-hand side of \eqref{3.cor} coincide
 up to order $\{\xi\}^{n+1}$ inclusively and the term $[\{j_\xi^n V(p,t)\}^{k-1},j_\xi^n W(p,t)]$ is
 treated analogously. The explicit expressions for these coefficients can be found using the
  formulas for higher order chain rule (Faa di Bruno type formulas, see e.g. \cite{faa,hajo}),
  but these expressions are lengthy and not essential for what follows, so we omit them.
 \par
 Note also that the
truncated jets $j_\xi^n V(p,t)$ are taken from the induction assumption. We seek for the solution
 of equation \eqref{3.jeteq} belonging to $L^2_{e^{n\theta_n+\theta_{n+1}}}(\R_-,H)$ for some $\theta_{n+1}$
 satisfying \eqref{2.gap} with $N$ replaced by $N_{n+1}$. Expanding \eqref{3.Fn} in series with respect to $\xi$,
 we get the recurrent equations for finding the
  "derivatives" $W^{(k)}_\xi(p,t):=W^{(k)}(p,t)[\{\xi\}^k]$:
\begin{multline}\label{3.rec}
\Dt W^{(k)}_\xi+AW^{(k)}_\xi-F'(W(p))W^{(k)}_\xi=\\=
\Phi(j_\xi^{k-1}W,j^{k-1}_\xi V),\  {\bf P}_{N_{n+1}}W^{(k)}_\xi\big|_{t=0}=0
\end{multline}
for $k\ge2$, where $\Phi$ is polynomial of order $k$ in $\xi$ which does not
contain $W^{(l)}_\xi$ with $l\ge k$. Thus, the functions $W^{(k)}_\xi$ can be, indeed, found recursively.
Moreover, the spectral gap assumption \eqref{4.gapnp1} guarantees that we can find $\theta_{n+1}$
satisfying \eqref{2.gap} with $N=N_{n+1}$ such that $\theta_{n+1}+n\theta_n$ also satisfies
 this condition. Therefore, Proposition \ref{Prop2.var} guarantees the existence and uniqueness of the
  homogeneous polynomials $W^{(k)}_\xi(p)$ satisfying
  \begin{equation}\label{3.W}
\|W^{(k)}_\xi(p)\|_{L^2_{e^{(\theta_{n+1}+k\theta_n)t}}(\R_-,H)}\le C\|\xi\|^k,
  \end{equation}
  for $k=1,\cdots,n+1$.
\par
To complete the proof of the theorem, we only need to verify that the jet $J_\xi W(p,t)$ satisfies
the compatibility conditions of order $n+1$. If this is verified, the rest of the proof coincides
 with the one given above for the case $n=2$. We postpone this verification till the next section.
 Thus, the theorem is proved by modulo of compatibility conditions.
\end{proof}

\begin{corollary}\label{Cor3.track} Let the assumptions of Theorem \ref{Th3.main} hold with $\mu>0$
 being small enough. Then the invariant manifold ${\bf P}_{N_{n}}\Cal M_{N_1}$ of the extended IF
  \eqref{3.IF-ext} possesses an exponential tracking property in $H_{N_{n}}$, i.e. for every
  solution $u_{N_{n}}(t)$ of \eqref{3.IF-ext} there exists the corresponding solution
  $\bar u_{N_{n}}$ belonging to this manifold such that
  \begin{equation}\label{3.track-red}
\|u_{N_{n}}(t)-\bar u_{N_{n}}(t)\|\le Ce^{-\theta_1 t}
  \end{equation}
  for some positive $C$ and $\theta_1$.
\end{corollary}
\begin{proof} As we have already mentioned, this is the standard corollary of the fact that $\Cal M_{N_1}$ is
normally hyperbolic and, therefore, persists under small $C^1$-perturbations, see \cite{bates,Fen72,Hirsh,katok} and references therein. Nevertheless, for
 the convenience of the reader, we sketch below a direct proof without formal refereing to
  normal hyperbolicity.
\par
We first construct an invariant manifold $\bar{\Cal  M}_{N_1}$ with the base $H_{N_1}$ in $H_{N_{n}}$ for the extended IF.
 We do this exactly as in the proof of Theorem \ref{Th2.IM} by solving the backward problem
 \begin{equation}\label{3.back}
\Dt u_{N_{n}}+Au_{N_{n}}-{\bf P}_{N_{n}}F(u_{N_{n}}+\tilde M_{N_{n}}(u_{N_{n}}))=0,\
{\bf P}_{N_1}u_{N_{n}}=p
 \end{equation}
 in the space $L^2_{e^{\theta t}}(\R_-,H_{N_{n}})$ with $\theta=(\lambda_{N_1}+\lambda_{N_1+1})/2$.
  This equation is $(C\mu)$-closed to
 \begin{equation}\label{3.back1}
\Dt \bar u_{N_{n}}+A\bar u_{N_{n}}-{\bf P}_{N_{n}}
F(\bar u_{N_{n}}+M_{N_{n}}(\bar u_{N_{n}}))=0,\
{\bf P}_{N_1}\bar u_{N_{n}}=p
 \end{equation}
 in the $C^1$-norm (since $\tilde M_{N_{n}}$ is
  $\mu$-closed to $M_{N_{n}}$ due to Theorem \ref{Th3.main}). Thus, using Remark \ref{Rem2.strange}
  and the Banach contraction theorem, we can construct a unique solution $u_{N_{n}}(t)$ of \eqref{3.back}
  in the $(C\mu)$-neighbourhood of the corresponding solution $\bar u_{N_{n}}$ of problem
   \eqref{3.back1} and vice versa. This gives us the existence of the manifold $\bar{\Cal M}_{N_1}$ which
   is generated by all backward solutions of \eqref{3.back1} belonging to the space
   $L^2_{e^{\theta t}}(\R_-,H_{N_{n}})$. Since the solutions belonging to the invariant
    manifold ${\bf P}_{N_{n}}\Cal M_{N_1}$
    satisfy exactly the same property, we conclude that
    $\bar{\Cal M}_{N_1}={\bf P}_{N_{n}}\Cal M_{N_1}$.
\par
It remains to verify that the manifold $\bar{\Cal M}_{N_1}$ possesses an exponential tracking property. This can
be done also as in the proof of Theorem \ref{Th2.IM} by considering the analogue of equation \eqref{2.v}
 for system \eqref{3.IF-ext} and using again that $\tilde M_{N_{n}}$ is close to $M_{N_{n}}$ in
 the $C^1$-norm. This finishes the proof of the corollary.
\end{proof}

\begin{corollary}\label{Rem3.H2} Arguing as in Corollary \ref{Cor2.L2}, we check that the
extended IM $\widetilde{\Cal M}_{N_n}$ is also $C^{n,\eb}$-submanifold of $H^2:=D(A)$.
\end{corollary}

\section{Verifying the compatibility conditions}\label{s4}
The aim of this section is to show that the  jets $J_\xi^{n+1} W(p,t)$, $p\in{\bf P}_{N_{n+1}}H$, constructed
 via \eqref{3.jeteq}, satisfy the
 compatibility conditions up to order $n~+~1$ and, thus, to complete the proof of Theorem \ref{Th3.main}.
 We will proceed by induction with respect to the order $m\le n+1$.
 \par
 Indeed, the first order compatibility conditions are trivially satisfied since the functions
 $W(p,t)$ are $C^{1,\eb}$-smooth. Assume that the $m$th order conditions
 are satisfied for some $m\le n+1$ and for all $m_1\le m$
\begin{equation}\label{4.jms}
  J_\xi^{m_1} W(p_1)-J^{m_1}_{\delta+\xi}W(p)=
  O_{\theta_{n+1}+(m_1-1)\theta_n}\((\|\delta\|+\|\xi\|)^{m_1+\eb}\),
\end{equation}
for all $\xi\in H$, $p_1,p\in{\bf P}_{N_{n+1}}\Cal M_{N_1}$, $\eb>0$, $\delta:=p_1-p$ and some constant $C$ which
 is independent of $p,p_1$.   Using the fact that $V(p,t)=W(p,t)$ for all
 $p\in{\bf P}_{N_{n+1}}\Cal M_{N_1}$ together with the analogue of \eqref{4.jms} for the already
  constructed jets $J_\xi^m V(p,t)$, we end up with
\begin{multline}\label{4.indd}
  V(p_1)=W(p_1)=V(p)+j^{m_1}_\delta V(p)+O_{m_1\theta_n+\eb}(\|\delta\|^{m_1+\eb})=\\
  =W(p)+j^{m_1}_\delta W(p)+O_{\theta_{n+1}+(m_1-1)\theta_n+\eb}(\|\delta\|^{m_1+\eb})
\end{multline}
  for all $p_1,p\in {\bf P}_{N_{n+1}}\Cal M_{N_1}$, $\delta:=p_1-p$ and, therefore
   $v(t):=V(p_1,t)-V(p,t)$ satisfies
\begin{multline}\label{4.VW}
v=j^{m_1}_\delta V(p)+O_{m_1\theta_n+\eb}(\|\delta\|^{m_1+\eb})=\\=j^{m_1}_\delta W(p)+
O_{\theta_{n+1}+(m_1-1)\theta_n+\eb}(\|\delta\|^{m_1+\eb}),\\
j^{m_1}_\delta V(p)-j^{m_1}_\delta W(p)=O_{\theta_{n+1}+(m_1-1)\theta_n+\eb}(\|\delta\|^{m_1+\eb}).
\end{multline}
We now turn to the $(m+1)$th-jets and start with the following lemma which gives the
compatibility conditions in the particular case $\xi=0$.

\begin{lemma} Let the above assumptions hold. Then
\begin{equation}\label{4.W0}
v=W(p_1)-W(p)=j^{m+1}_\delta W(p)+O_{\theta_{n+1}+m\theta_n+\eb}(\|\delta\|^{m+1+\eb}),
\end{equation}
for all $p_1,p\in {\bf P}_{N_{n+1}}\Cal M_{N_1}$ and $\delta:=p_1-p$. Moreover,
\begin{equation}\label{4.F0}
F(V(p_1))=F^{[m+1]}(p,\delta)+O_{\theta_{n+1}+m\theta_n+\eb}(\|\delta\|^{m+1+\eb})
\end{equation}
for some $\eb>0$.
\end{lemma}
\begin{proof} Let $R:=v-j_\delta^{m+1}W(p)$. Then, by the definition \eqref{3.jeteq}, this function solves
\begin{equation}
\Dt R+A R=F(V(p_1))-F^{[m+1]}(p,\delta),\ \ {\bf P}_{N_{n+1}}R\big|_{t=0}=0.
\end{equation}
Let us study the term $F^{[m+1]}(p,\delta)$ at the right-hand side (which is defined by \eqref{3.Fn}).
 Using  \eqref{4.VW} and the trick \eqref{3.cor},
we may replace $j_\delta^{m} V(p)$ and $j_\delta^{m} W(p)$ by $v$ in all terms in \eqref{3.Fn} which contain
the second and higher derivatives of $F$ (the error
will be of order $\|\delta\|^{m+1+\eb}$). Actually, we cannot do this in the term with the
first derivative at the moment since this requires \eqref{4.VW} for $W$ of order $m+1$ which
 we are now verifying. This, gives
\begin{multline}\label{4.F2m}
F^{[m+1]}(p,\delta)=F(V(p))+\\+F'(V(p))j^{m+1}_\delta W(p)+\sum_{k=2}^{m+1}\frac1{k!}F^{(k)}(V(p))[\{v\}^k]
+O_{\theta_{n+1}+m\theta_n+\eb}(\|\delta\|^{m+1+\eb}).
\end{multline}
 Indeed,  let us consider the terms in \eqref{3.Fn} containing $j^m_\delta W$ only (the terms without it are analogous, but simpler). Using the analogue of \eqref{3.cor}:
 \begin{multline}
 [\{j^m_\delta V (p)\}^{k-1}, j^m_\delta W(p)]\rightarrow\\\rightarrow\sum_{\substack{n_1+\cdots +n_k\le m+1\\n_i\in\Bbb N}}B'_{n_1,\cdots,n_k}\{j^{n_1}_\delta V(p),\cdots,j^{n_{k-1}}_\delta V(p), j_\delta^{n_k} W(p)\},
 \end{multline}
 the growth exponent of the remainder does not exceed
 $$
 \(n_1+\cdots+n_{k-1}\)\theta_n+\theta_{n+1}+(n_k-1)\theta_n+\eb\le
 \theta_{n+1}+m\theta_n+\eb,
 $$
 where we have implicitly used our induction assumptions \eqref{4.VW} and decreased the exponent $\eb$ if necessary.
 \par
 Using now the Taylor theorem for $F\in C^{m+1,\eb}$ together with estimate \eqref{2.b-est} for $v$,
  we infer that
 $$
 F(V(p_1))-F^{[m+1]}(p,\delta)=F'(V(p))R+O_{\theta_{n+1}+m\theta_n+\eb}\(\|\delta\|^{m+1+\eb}\)
 $$
 and, therefore, the function $R$ solves
 \begin{equation}\label{4.R0}
\Dt R+AR-F'(V(p))R=O_{\theta_{n+1}+m\theta_n+\eb}(\|\delta\|^{m+1+\eb}), \
\ {\bf P}_{N_{n+1}}R\big|_{t=0}=0.
 \end{equation}
Since by the induction assumption $\theta_n<\lambda_{N_n+1}-L$, assumption \eqref{4.gapnp1} guarantees
 the existence of $\theta_{n+1}$  and $\eb>0$ such that $\theta_{n+1}+m\theta_n+\eb$ satisfies \eqref{2.gap}
 with $N$ replaced by $N_{n+1}$. Thus, Proposition \ref{Prop2.var} gives the estimate
 $$
 \|R\|_{L^2_{e^{(\theta_{n+1}+m\theta_n+\eb)t}}(\R_-,H)}\le C\|\delta\|^{m+1+\eb}
 $$
 and \eqref{4.W0} is proved. Estimate \eqref{4.F0} is now a straightforward corollary of
 \eqref{4.F2m} and the Taylor theorem
  (since we are now allowed to replace $j_\delta^{m+1}W$ by $v$). Thus, the Lemma is proved.
\end{proof}
We now turn to the general case $\xi\ne0$. To this end we need the following key lemma.

\begin{lemma}\label{Lem4.key} Let the above assumptions hold. Then, the following formula is satisfied:
\begin{multline}\label{4.key}
F^{[m+1]}(p_1,\xi)-F^{[m+1]}(p,\xi+\delta)=\\F'(V(p))\(j_\delta^{m+1}W(p)+j^{m+1}_\xi W(p_1)-
j_{\xi+\delta}^{m+1}W(p)\)+\\+O_{\theta_{n+1}+m\theta_n+\eb}\((\|\delta\|+
\|\xi\|)^{m+1+\eb}\),
\end{multline}
where $\xi\in H$, $p_1,p\in{\bf P}_{N_{n+1}}\Cal M_{N_1}$ and $\delta=p_1-p$.
\end{lemma}
\begin{proof} Indeed, according to the definition \eqref{3.Fn} and formula \eqref{4.F0}, we have
\begin{multline}
F^{[m+1]}(p_1,\xi)=F^{[m+1]}(p,\delta)+F'(V(p_1))j_\xi^{m+1}W(p_1)+\\+
\sum_{l=2}^{m+1}\frac1{l!}\(lF^{(l)}(V(p_1))[j^m_\xi W(p_1),\{j_\xi^{m}V(p_1)\}^{l-1}]-\right.\\\left.
(l-1)F^{(l)}(V(p_1))[\{j_\xi^{m}V(p_1)\}^{l}]\)+
O_{\theta_{n+1}+m\theta_n+\eb}\((\|\xi\|+\|\delta\|)^{m+1+\eb}\).
\end{multline}
We recall that, according to our agreement and formulas \eqref{3.cor},  the right-hand side does
 not contain the terms of order larger than $m+1$. Expanding now the derivatives $F^{(l)}(V(p_1))$
 into Taylor series around $V(p)$ and using \eqref{4.VW}, we get
 \begin{multline}
F^{[m+1]}(p_1,\xi)=F^{[m+1]}(p,\delta)+F'(V(p))(j_\xi^{m+1}W(p_1)-j^m_\xi W(p_1))+\\+
\sum_{l=1}^{m+1}\sum_{k=l}^{m+1}\frac1{l!(k-l)!}\(lF^{(k)}(V(p))
[\{j^m_\delta V(p)\}^{k-l},j^m_\xi W(p_1),\{j^m_\xi V(p_1)\}^{l-1}]-\right.\\\left.(l-1)F^{(k)}(V(p))
[\{j^m_\delta V(p)\}^{k-l},\{j^m_\xi V(p_1)\}^{l}]\)+\\+O_{\theta_{n+1}+m\theta_n+\eb}\((\|\xi\|+\|\delta\|)^{m+1+\eb}\).
 \end{multline}
Finally, changing the order of summation, we arrive at
\begin{multline}\label{4.rep}
F^{[m+1]}(p_1,\xi)=F^{[m+1]}(p,\delta)+F'(V(p))j_\xi^{m+1}W(p_1)+\\+
\sum_{k=2}^{m+1}\frac1{k!}\sum_{l=1}^k C^l_k\(lF^{(k)}(V(p))
[\{j^m_\delta V(p)\}^{k-l},j^m_\xi W(p_1),\{j^m_\xi V(p_1)\}^{l-1}]-\right.\\\left.(l-1)F^{(k)}(V(p))
[\{j^m_\delta V(p)\}^{k-l},\{j^m_\xi V(p_1)\}^{l}]\)+\\+O_{\theta_{n+1}+m\theta_n+\eb}\((\|\xi\|+\|\delta\|)^{m+1+\eb}\)
\end{multline}
Let us now look to the term $F^{[m+1]}(p,\xi+\delta)$. According to \eqref{3.Fn}, we have
\begin{multline}\label{4.prep}
F^{[m+1]}(p,\xi+\delta)=F(V(p))+F'(V(p))j_{\xi+\delta}^{m+1}W(p)+\\+
\sum_{k=2}^{m+1}\frac1{k!}\(kF^{(k)}(V(p))[j^m_{\xi+\delta} W(p),\{j_{\xi+\delta}^{m}V(p)\}^{k-1}]-\right.\\\left.
(k-1)F^{(k)}(V(p))[\{j_{\xi+\delta}^{m}V(p)\}^{k}]\).
\end{multline}
From the induction assumption, the compatibility assumptions \eqref{4.jms} hold for $j_{\xi+\delta}^{m_1} W$
 and  give
$$
j_{\xi+\delta}^{m_1}W(p)=j_\delta^{m_1} W(p)+j_\xi^{m_1} W(p_1)+O_{\theta_{n+1}+(m_1-1)\theta_n+\eb}
\((\|\delta\|+\|\xi\|)^{m_1+\eb}\)
$$
for all $m_1\le m$ and the analogous identities hold also for $j_{\xi+\delta}^{m_1} V$:
$$
j_{\xi+\delta}^{m_1}V(p)=j_\delta^{m_1} V(p)+j_\xi^{m_1} V(p_1)+O_{m_1\theta_{n}+\eb}
\((\|\delta\|+\|\xi\|)^{m_1+\eb}\).
$$
Moreover, using \eqref{4.VW}, we may also get
$$
j_{\xi+\delta}^{m_1}W(p)=j_\delta^{m_1} V(p)+j_\xi^{m_1} W(p_1)+O_{\theta_{n+1}+(m_1-1)\theta_n+\eb}
\((\|\delta\|+\|\xi\|)^{m_1+\eb}\)
$$
for all $m_1\le m$. Inserting these formulas to \eqref{4.prep}, we arrive at
\begin{multline}\label{4.prep1}
F^{[m+1]}(p,\xi+\delta)=F(V(p))+F'(V(p))j_{\xi+\delta}^{m+1}W(p)+\\+
\sum_{k=2}^{m+1}\frac1{k!}\(kF^{(k)}(V(p))[j^m_{\delta} V(p)+j^m_{\xi} W(p_1),
\{j^m_{\delta} V(p)+j^m_{\xi} V(p_1)\}^{k-1}]-\right.\\\left.
(k-1)F^{(k)}(V(p))[\{j^m_{\delta} V(p)+j^m_{\xi} V(p_1)\}^{k}]\)+\\+O_{\theta_{n+1}+m\theta_{n}+\eb}
\((\|\delta\|+\|\xi\|)^{m+1+\eb}\).
\end{multline}
Using the binomial formula \eqref{2.1}, we arrive at
\begin{multline}\label{4.prep2}
F^{[m+1]}(p,\xi+\delta)=F(V(p))+F'(V(p))j_{\xi+\delta}^{m+1}W(p)+\\+
\sum_{k=2}^{m+1}\frac1{k!}\(\sum_{l=1}^{k}k C_{k-1}^{l-1} F^{(k)}(V(p))[j^m_{\xi} W(p_1),
\{j^m_{\delta} V(p)\}^{k-l},\{j^m_{\xi} V(p_1)\}^{l-1}]+\right.\\\left.+\sum_{l=0}^{k-1}k C_{k-1}^{l}
F^{(k)}(V(p))[j^m_{\delta} V(p),
\{j^m_{\delta} V(p)\}^{k-l-1},\{j^m_{\xi} V(p_1)\}^{l}]-\right.\\\left.-
\sum_{l=0}^k(k-1)C_k^lF^{(k)}(V(p))[\{j^m_{\delta} V(p)\}^{k-l},\{j^m_{\xi} V(p_1)\}^{l}]\)+\\+O_{\theta_{n+1}+m\theta_{n}+\eb}.
\((\|\delta\|+\|\xi\|)^{m+1+\eb}\).
\end{multline}
We need to compare \eqref{4.rep} and \eqref{4.prep2}. To this end, we first note that
$$
lC_k^l=kC_{k-1}^{l-1}
$$
and, therefore, the terms containing the jets of $W$ in these two formulas coincide. Thus,
we only need to look at the terms without jets of $W$. In the case $l=k$, we have only one term in the right-hand side of
 \eqref{4.prep2} which obviously coincides with the analogous term in \eqref{4.rep}. Let us now look at
 the terms with $l=1,\cdots,k-1$. Due to the obvious identity
$$
-(l-1)C^l_k=kC^{l}_{k-1}-(k-1)C^l_k,
$$
these terms again coincide. Thus, it remains to look at the extra terms which correspond to $l=0$
 in \eqref{4.prep2} and which are absent in the sums of \eqref{4.rep}. Finally, using
  \eqref{4.VW} and \eqref{4.F0},
we get the following identity involving these extra terms:
\begin{multline}
F(V(p))+\sum_{k=2}^{m+1}\frac1{k!}F^{(k)}(V(p))[\{j^{m}_\delta V(p)\}^k]=
F^{[m+1]}(p,\delta)-\\-F'(V(p))j_\delta^{m+1}W(p)+
O_{\theta_{n+1}+m\theta_{n}+\eb}
\((\|\delta\|+\|\xi\|)^{m+1+\eb}\).
\end{multline}
This gives the identity
\begin{multline}
F^{[m+1]}(p_1,\xi)-F'(V(p))j^{m+1}_\xi W(p_1)=F^{[m+1]}(p,\xi+\delta)-\\-
F'(V(p))\(j_{\xi+\delta}^{m+1}W(p)-j_\delta^{m+1}W(p)\)+
O_{\theta_{n+1}+m\theta_{n}+\eb}
\((\|\delta\|+\|\xi\|)^{m+1+\eb}\)
\end{multline}
and finishes the proof of the lemma.
\end{proof}
We are now ready to finish the check of the compatibility conditions. Note that, due to \eqref{4.W0}, we have
\begin{multline}
J^{m+1}_\xi W(p_1)-J^{m+1}_{\xi+\delta}W(p)=\\j_\delta^{m+1}W(p)+j^{m+1}_\xi W(p_1)-
j_{\xi+\delta}^{m+1}W(p)+O_{\theta_{n+1}+m\theta_n+\eb}\((\|\delta\|+\|\xi\|)^{m+1+\eb}\).
\end{multline}
Let finally $U(t):=J^{m+1}_\xi W(p_1)-J^{m+1}_{\xi+\delta} W(p)$. Then, according to  definition \eqref{3.Wjet}, Lemma
\ref{Lem4.key} and the fact that $\delta=p_1-p$, this function solves the equation
\begin{multline}
\Dt U+A U-F'(V(p))U=\\=
O_{\theta_{n+1}+m\theta_n+\eb}\(\|\delta\|+\|\xi\|)^{m+1+\eb}\),\ \ {\bf P}_{N_{n+1}}U\big|_{t=0}=0
\end{multline}
and by Proposition \ref{Prop2.var}, we arrive at
\begin{equation}
J^{m+1}_\xi W(p_1)-J^{m+1}_{\xi+\delta} W(p)=O_{\theta_{n+1}+m\theta_n+\eb}\((\|\delta\|+\|\xi\|)^{m+1+\eb}\).
\end{equation}
Thus, the $(m+1)$th order compatibility conditions for $J^{m+1}_\xi W(p)$ are verified. The induction with respect
 to $m$ gives us that $J_\xi^{n+1} W(p)$ also satisfies the compatibility conditions (of course, we cannot
  take $m>n$ since we need the compatibility conditions of order $m$ for $J^m_\xi V(p)$ to proceed).
  This completes the proof of our main Theorem \ref{Th3.main}.

\section{Examples and concluding remarks}\label{s5}
In this section we give several examples to the proved main theorem as well as its reinterpretations
 and state some interesting problems for further study. We start with the application to 1D
  reaction-diffusion equation.
\begin{example}\label{Ex6.1} Let us consider the following reaction-diffusion system in 1D domain
 $\Omega=(-\pi,\pi)$:
\begin{equation}\label{5.RDS}
\Dt u=a\partial_x^2 u-f(u),\ \ u\big|_{\Omega}=0,\ \ u\big|_{t=0}=u_0,
\end{equation}
where $u$ is an unknown function, $a>0$ is a given
viscosity parameter,
and $f(u)$ is a given smooth function satisfying $f(0)=0$ and  some
 dissipativity conditions, for instance,
  $$
  f(u)u\ge-C+\alpha|u|^2,\  u\in\R.
  $$
for some $C$ and $\alpha>0$ (e.g. $f(u)=u^3-u$ as in the case of real Ginzburg-Landau equation). Then,
due to the maximum principle, we have the following dissipative estimate for the solutions of \eqref{5.RDS}:
\begin{equation}\label{5.inf}
\|u(t)\|_{L^\infty}\le \|u_0\|_{L^\infty}e^{-\alpha t}+C_*,
\end{equation}
where the constant $C_*$ is independent of $u_0$, see, e.g. \cite{BV92,CV02,T97}. Thus, the associated
 solution semigroup $S(t)$ acting in the phase space $H:=H^1_0(\Omega)$ possesses an absorbing set in $C(\bar\Omega)$ and cutting-off the nonlinearity outside of this ball, we may assume without
  loss of generality that $f\in C_0^\infty(\R)$.
  \par
  After this transform, equation \eqref{5.RDS} can be considered as an abstract parabolic equation
   \eqref{3.1} in the Sobolev space $H=H^1_0(\Omega)$. Since this space is an algebra with respect to
    point-wise multiplication (since we have only one spatial variable), the corresponding non-linearity
    $F(u)(x):=f(u(x))$ is $C^\infty$-smooth and all its derivatives are globally bounded.
    \par
    Finally, the linear operator $A$ in this example is $A=-a\partial_x^2$ endowed with
    the Dirichlet boundary conditions. Obviously, this operator is self-adjoint, positive definite
    and its inverse is compact. Moreover, its eigenvalues
    $$
    \lambda_k=a k^2,\ \ k\in\Bbb N
    $$
 satisfy \eqref{3.sginf}. Thus, our main theorem \ref{Th3.main} is applicable here and, therefore, problem
 \eqref{5.RDS} possesses an IM $\Cal M_{N_1}$ of smoothness $C^{1,\eb}$ for some $\eb>0$ and, for every
 $n\in\Bbb N$, this IM can be extended to a manifold $\widetilde {\Cal M}_{N_n}$ of regularity $C^{n,\eb_n}$,
 $\eb_n>0$, in the sense of Definition \ref{Def3.ext-IF}.
\end{example}
\begin{remark}\label{Rem5.gen} Our general theorem is applicable not only for a scalar
re\-acti\-on-diffusion equation \eqref{5.RDS}, but also for systems where the analogue of \eqref{5.inf} is known,
for instance, for the case of 1D complex Ginzburg-Landau equation (however, one should be careful in the
 case where the diffusion matrix is not self-adjoint and especially when it contains non-trivial
  Jordan cells. In this case, even Lipschitz IM may not exist, see \cite{kwa} for more details).
\par
A bit unusual choice of the phase space $H=H^1_0(\Omega)$ (instead of the natural one $H=L^2(\Omega)$
is related with the fact that we need  $H$ to be an algebra in order to define Taylor jets for
 the nonlinearity $F$ and to verify that it is $C^\infty$. This however may be relaxed in
  applications since backward solutions of \eqref{3.IF} and \eqref{2.var} are usually smooth
   in space and time if the non-linearity $f$ is smooth, so the Taylor jets for $V(p,t)$ will be
    well-defined even if we consider $L^2(\Omega)$ as a phase space and the theory works with
     minimal changes. This observation may be useful if we want to remove the assumption $f(0)=0$
     in \eqref{5.RDS}, but in order to avoid technicalities, we prefer not to go further in this direction here.
     \par
  The restriction to 1D case is motivated by the fact that the spectral gap condition
   \eqref{3.sginf} is naturally satisfied by the Laplacian in 1D case only (it is an open problem already in 2D case).
\par
   If we consider higher-order operators, say bi-Laplacian then the analogous result holds also in 3D. The typical example
   here is given by Swift-Hohenberg equation in a bounded domain $\Omega\subset\R^3$:
   $$
   \partial_t u=-(\Delta+1)^2 u+u-u^2,\  u\big|_{\partial\Omega}=\Delta u\big |_{\partial\Omega}=0,
   $$
   where the spectral gap condition \eqref{3.sginf} is also satisfied, see \cite{Z14}
   and references therein. We also note that although our main theorem is stated and proved
   for the case where $F$ maps $H$ to $H$, it can be generalized in a very straightforward way to the
   case where the operator $F$ decreases smoothness and maps $H$ to $H^{-s}:=D(A^{-s/2})$
   for some $s\in(0,2)$.
   The spectral gap assumption \eqref{3.sginf} should be replaced by
   $$
   \limsup_{n\to\infty}
   \left\{\frac{\lambda_{n+1}-\lambda_n}{\lambda_{n+1}^{s/2}+\lambda_n^{s/2}}\right\}=\infty.
   $$
After this extension, our theorem becomes applicable to equations which contain spatial derivatives in the
 non-linearity. Typical example of such applications is 1D Kuramoto-Sivashinski equation
 $$
 \partial_t u+a\partial_x^2 u+\partial^4_x u+u\partial_x u=0, \ \Omega=(-\pi,\pi),\ \ a>0
 $$
 endowed with Dirichlet or periodic boundary conditions, see \cite{Z14} for more details.
\end{remark}
\begin{remark}\label{Rem5.spa} As we have mentioned in the introduction, there is an essential recent
 progress in constructing  IMs for concrete classes of parabolic equations
 which do not satisfy the spectral gap conditions (such as scalar reaction-diffusion equations
  in higher dimensions, 3D Cahn-Hilliard or complex Ginzburg-Landau equations,
  various modifications of Navier-Stokes systems, 1D reaction-diffusion-advection systems, etc.). The techniques developed in our paper
   is not directly applicable to such problems (in particular, our technique is strongly based on the
    Perron method of constructing the IMs and it is not clear how to use
    the Perron method here since we do not have the so-called absolute normal hyperbolicity in the most
     part of equations mentioned above, see \cite{K18,KZ15} for more details). However, we believe that the
      proper modification of our method would allow to cover these cases as well. We return to this
       problem somewhere else.
\end{remark}
We now give an alternative (probably more transparent and more elegant) formulation of Theorem \ref{Th3.main}. We
recall that in Theorem \ref{Th3.main}, we have directly constructed a smooth extended IF \eqref{3.IF-ext}
for the initial equation \eqref{3.1}. This extended IF captures all non-trivial
 dynamics of \eqref{3.1}, but the associated smooth extended IM ${\Cal M}_n$ is not associated
  with the "true" IM of any system of the form \eqref{3.1}. This drawback can be easily corrected
   in more or less standard way which leads to the following reformulation of our main result.
   \begin{corollary}\label{Cor5.eq} Let the assumptions of Theorem \ref{Th3.main} be satisfied and let $\Cal M_{N_1}$
   be the $C^{1,\eb_1}$-smooth IM of equation \eqref{3.1} which corresponds to the first spectral gap.
   Then, for every $n\in\Bbb N$, $n>1$, there exists a modified nonlinearity $\widetilde F: H\to H$
   which belongs to $C^{n-1,\eb_n}_b(H,H)$ for some $\eb_n>0$ such that
\par
1) The initial IM $\Cal M_{N_1}$ is simultaneously an IM for the modified equation
\begin{equation}\label{5.maineq}
\Dt u+Au=\widetilde F_n(u).
\end{equation}

\par
2) Equation \eqref{5.maineq} possesses a $C^{n,\eb_n}$-smooth IM $\widetilde{\Cal M}_{N_n}$ of dimension
 $N_n$ such that the initial IM $\Cal M_1$ is a normally hyperbolic globally
 stable submanifold of $\widetilde{\Cal M}_{N_n}$.
\par
3) The nonlinearity $\widetilde F_n(u)$ depends on the variable $u_{N_n}:={\bf P}_{N_n}u$ only and
the IF associated with the IM $\widetilde{\Cal M}_{N_n}$ is given by \eqref{3.IF-ext} where $K_2$ is replaced by $N_n$.
   \end{corollary}
\begin{proof} Indeed, we take the manifold $\widetilde{\Cal M}_{N_n}$ constructed in Theorem \ref{Th3.main}
and define the desired function $\widetilde F_n$ as follows
\begin{equation}\label{5.PN}
{\bf P}_{N_n}\widetilde F_n(u):={\bf P}_{N_n}F(u_{N_n}+\widetilde{M}_{N_n}(u_{N_n}))
\end{equation}
and
\begin{multline}\label{5.QN}
{\bf Q}_{N_n}\widetilde F_n(u):=\\=\widetilde{M}_{N_n}'(u_{N_n})
[-A\widetilde{M}_{N_n}(u_{N_n})+{\bf P}_{N_n}F(u_{N_n}+
\widetilde{M}_{N_n}(u_{N_n}))]+A\widetilde{M}_{N_n}(u_{N_n}).
\end{multline}
Then, due to the choice of ${\bf P}_{N_n}$-component of $\widetilde F_n(u)$, the
equation for $u_{N_n}$ is decoupled from the equation for the ${\bf Q}_{N_n}$-component and coincides
with the extended IF for \eqref{5.maineq} constructed in Theorem \ref{Th3.main}. On the other hand, the
${\bf Q}_{N_n}$-component of $\widetilde F_n$ is chosen in a form which guarantees that $\widetilde{\Cal M}_{N_n}$
is an invariant manifold for equation \eqref{5.maineq}. Moreover, if $u(t)$ solves equation
\eqref{5.maineq} with such a nonlinearity and $v(t):=u(t)-{\bf P}_{N_n}u(t)-\widetilde M_{N_n}(u_{N_n}(t))$,
then this function satisfies
$$
\Dt v+Av=0,\ \ {\bf P}_{N_n}v(t)\equiv0
$$
and, therefore,
$$
\|v(t)\|_H\le \|v(0)\|_He^{-\lambda_{N_{n+1}}t}.
$$
Thus, $\widetilde M_{N_n}$ is indeed an IM for problem \eqref{5.maineq} and we only need to check
 the regularity of the modified function $\widetilde F_n$.
\par
The ${\bf P}_{N_n}$ component \eqref{5.PN} is clearly $C^{n,\eb_n}$-smooth, but
the situation with the ${\bf Q}_{N_n}$ is a bit more delicate due to the presence of terms $A\widetilde M_{N_n}(u_{N_n})$ and
$\widetilde M'_{N_n}(u_{N_n})$. The first term is not dangerous since we know that
$\widetilde M_{N_n}$ is $C^{n,\eb_n}$-smooth as the map from $H_{N_n}$ to $H^2$. The second term is
 worse and decreases the smoothness of the $\widetilde F_n$ till
  $C^{n-1,\eb_n}$. Thus, the corollary is proved.
\end{proof}
\begin{remark} The modified non-linearity $\tilde F_n(u)$ can be interpreted as a "clever" cut-off of
 the initial non-linearity $F(u)$ outside of the global
  attractor (even outside of the IM of minimal dimension). In this sense we may say that all
  obstacles for the existence of $C^{n,\eb}$-smooth IM can be removed by the appropriate cutting
   off the nonlinearity outside of the global attractor which does not affect
    the dynamics of the initial problem. This demonstrates the importance of finding the proper cut
     off procedure in the theory of IMs.
  \end{remark}
\begin{example}\label{Ex6.sell} We now return to  the model example of G. Sell introduced in Example
 \ref{Ex2.sell} and show how the problem of smoothness of an invariant manifold can be resolved. Since the non-linearity for this system is not {\it globally} Lipschitz continuous, the above developed theory is formally not applicable and we need to cut-off the nonlinearity first. We overcome this problem by considering only {\it local} manifolds in a small neighbourhood of the origin.
 \par
 Indeed, it is not difficult
  to see that system \eqref{2.counter} has an explicit particular solution
  $$
  u_1(t)=\pm e^{-t},\ \ u_{n+1}(t)=C_ne^{-2^n t}t^{2^n-1},\ \ n>1,
  $$
  where the coefficients $C_n$ satisfy the recurrent relation
  $$
  C_{n+1}=\frac1{2^n-1}C_{n}^2,\ \ C_0=1.
  $$
  This solution determines 1D local invariant manifold
  $$
  \Cal M_1=\{p+M(p):\, p\in H_1=\R,\ \  |p|<\beta\},
  $$
   where $M:\R\to H$
  is defined by $M=(0,M_1(p),M_2(p),\cdots,)$ and
  $$
  M_{n+1}(p)=C_n p^{2^n}\(\ln\frac1{|p|}\)^{2^n-1},\ \ n\in\Bbb N
  $$
  which is 1D IM for system \eqref{2.counter} and $\beta$ is a sufficiently small positive number.
   Indeed, since $C_n\le 2^{-\alpha 2^n}$ for
   some positive $\alpha$, this manifold is well-defined  as a local submanifold of $H=l_2$ (if $\beta>0$ is small enough)
    and is $C^{1,\eb}$-smooth for any $\eb\in(0,1)$. Moreover, we see that $M_2(p)$
    is only $C^{1,\eb}$-smooth and higher components are
     more regular, in particular, $M_n(p)$ is $C^{2^{n-1}-1,\eb}$-smooth. This guesses us how to define
      the extended manifolds of an arbitrary finite smoothness. Namely, let us fix some $n\in\Bbb N$
      and consider the following manifold:
      \begin{multline}
\widetilde {\Cal M}_n:=\{p+\widetilde M_n(p),\ \ p\in H_n,\ \ |p_1|<\beta\},\\
\widetilde M_n(p):=(\{0\}^n,M_{n+1}(p_1), M_{n+2}(p_1), M_{n+3}(p_1),\cdots).
      \end{multline}
Clearly $\widetilde{\Cal M}_n$ is $C^{2^n-1,\eb}$-smooth and $\Cal M_1$ is a submanifold of
 $\widetilde{\Cal M}_n$. Moreover, if we define the modified non-linearity $\widetilde F_n(u)$ as follows:
 \begin{equation}
\widetilde F_n(u)=(0,u_1^2,u_2^2,\cdots, u_{n-1}^2,M_{n+1}(u_1),M_{n+2}(u_1),\cdots),
 \end{equation}
 then it will be $C^{2^n-1,\eb}$-smooth and the extended manifold $\widetilde{\Cal M}_n$
 will be an IM
 for the corresponding modified equation \eqref{5.maineq}. Finally, the normal
 hyperbolicity of $\Cal M_1$ in $\widetilde{\Cal M}_n$ follows from the fact that any solution on
  $\Cal M_1$ decays to zero not faster than  $e^{-t}$ due to the non-zero first component, if we look to the transversal directions, the smallest decay rate is determined by the second component and this decay is at least as $t^3 e^{-2t}$. Since our model system is explicitly
   solvable, we leave verifying of this normal hyperbolicity to the reader. We also note that
   the extended IF in this case reads
   $$
   \frac d{dt}u_1+u_1=0,\ \ \frac d{dt}u_k+2^{k-1}u_k=u_{k-1}^2,\ \ k=2,\cdots,n
   $$
   which is nothing more than the Galerkin approximation system to \eqref{2.counter}.
\end{example}
\begin{remark}\label{Rem5.final} We see that, in the toy example of equation \eqref{2.counter}, we can
find the desired extension of the initial IM explicitly without using the Whitney extension theorem
 (and even without assuming the global boundedness of $F$ and its derivatives). Moreover, the
 dependence of smoothness of the extended IM on its dimension is very nice, namely, if we want to
 have $C^n$-smooth IM, it is enough to take $\dim \widetilde M\sim \log_2 n$. Of course, this
  is partially related with good exponentially growing spectral gaps, but the main reason
  is that we have an extra regularity property for the initial IM, namely, that the smoothness
  of projections ${\bf Q}_k M(p)$ grows with $k$. Unfortunately, this is not true in a more or
   less general case which makes the extension construction much more involved. In particular, we do
    not know how to gain more than one unit of smoothness from one spectral gap and have to use $n$
    different spectral gaps to get $n$ units of smoothness. This, in turn, leads to extremely fast growth of
    the dimension of the manifold with respect to the regularity (as not difficult to see,
    in Example \ref{Ex6.1}, the dimension of $\widetilde{\Cal M}_{N_n}$ grows
     as a double exponent with respect to $n$).
\par
We believe that this problem is technical and the estimates for the dimension can be essentially improved.
 Indeed, if we would be able to get $n$ units of extra regularity using one extra
  (sufficiently large) gap the above mentioned growth of the dimension would become linear in $n$ in
   Example \ref{Ex6.1}. We expect that this linear growth is optimal, and even able to construct
    the corresponding Taylor jets. But these jets do not satisfy the compatibility conditions and we do not
     know how to correct them properly.
\end{remark}


\begin{thebibliography}{9}
\bibitem{BV92}
A. Babin and M. Vishik, {\it Attractors of evolution equations,} Studies in Mathematics and its
Applications, 25. North-Holland Publishing Co., Amsterdam, 1992.
\bibitem{bates} P. Bates, K. Lu and C. Zeng, {\it Persistence of Overflowing Manifolds for Semiflow.} Comm. Pure Appl. Math., vol. 52, (1999) 983--1046.
\bibitem{3}
A. Ben-Artzi, A. Eden, C. Foias, and B. Nicolaenko,
{\it H\"older continuity for the inverse of
Mane’s projection.}
J. Math. Anal. Appl., vol.178, (1993) 22--29.
\bibitem{CV02}
V. Chepyzhov and M. Vishik, {\it Attractors for equations of mathematical physics,} American Mathematical Society Colloquium Publications, 49. American Mathematical Society,
Providence, RI, 2002.
\bibitem{S-sell}
S.-N. Chow, K. Lu, and G. Sell,
{\it Smoothness of inertial manifolds},
Jour.  Math. Anal. and Appl.,
vol. 169, no. 1
(1992)  283--312.
\bibitem
{CFNT89}
 P. Constantin, C. Foias, B. Nicolaenko, and R. Temam,
 \emph{Inertial Manifolds for Dissipative Partial Differential Equations (Applied Mathematical Sciences,
  no. 70)}, Springer-Verlag, New York, 1989.
\bibitem{EKZ13}
A. Eden, V. Kalanarov and S. Zelik, {\it Counterexamples to the regularity of Mane projections and global
attractors,} Russian Math Surveys, vol. 68, no. 2, (2013) 199--226.
\bibitem{fef}
C. Fefferman, {\it A sharp form of Whitney's extension theorem}, Annals of Mathematics, vol. 161,  no.1,  (2005) 509--577.
\bibitem{Fen72}
N. Fenichel, {\it Persistence and smoothness of invariant manifolds for flows,} Indiana Univ.
Math. J., vol. 21, (1971/1972), 193--226.
\bibitem
{FST88}
 C. Foias, G. Sell, and  R. Temam,
 \emph{Inertial manifolds for nonlinear evolutionary equations},
 J. Differential Equations, vol. 73, no. 2, (1988) 309--353.
\bibitem
{GG18}
 C. Gal and Y. Guo,
 \emph{Inertial manifolds for the hyperviscous Navier-Stokes equations},
  J. Differential Equations, vol. 265, no. 9, (2018) 4335--4374.
\bibitem{hajo}
P. H\'ajek and Michal Johanis,
{\it Smooth Analysis in Banach Spaces},
In: De Gruyter Series in Nonlinear Analysis and Applications, 19, De Gruyter , 2014.
\bibitem{ha}
{\au J. Hale},
{\bk Asymptotic Behaviour of Dissipative Systems},
\eds{Math. Surveys and Mon.}{AMS Providence, RI}{1987}
\bibitem{hen} D. Henry, {\it Geometric theory of semilinear parabolic equations.} Lecture Notes in Mathematics,
840. Springer-Verlag, Berlin–New York, 1981.
\bibitem{Hirsh} M. Hirsch, C. Pugh, and M. Shub, {\it Invariant manifolds.} Lecture Notes in Mathematics, Vol.
583. Springer-Verlag, Berlin–New York, 1977.
\bibitem{hunt}
B. Hunt and V. Kaloshin,
{\it Regularity of embeddings of infinite-dimensional fractal sets
into finite-dimensional spaces.}
Nonlinearity,  vol. 12, (1999) 1263--1275.
\bibitem{katok}
A. Katok and B. Hasselblatt,
{\it Introduction to the modern theory of dynamical systems.}
Encyclopedia of Mathematics and its Applications, 54. Cambridge University Press, Cambridge, 1995.
\bibitem{kie}
H. Kielh\"ofer, {\it Bifurcation Theory: An Introduction with Applications to Partial Differential Equations}, Applied Mathematical Sciences 156, Springer-Verlag New York,  2012.
\bibitem{kok} N. Koksch, {\it Almost Sharp Conditions for the Existence of Smooth Inertial manifolds.} in: Equadiff 9: Conference on Differential Equations and their Applications : Proceedings, edited by Z. Dosla, J. Kuben, J. Vosmansky, Masaryk University, Brno, 1998, 139--166.
\bibitem
{K18}
 A. Kostianko,
  \emph{Inertial manifolds for the 3D modified-Leray-$\alpha$ model with periodic boundary conditions},
   J. Dyn. Differ. Equations, vol. 30, no. 1, (2018) 1--24.
\bibitem{K20}
A. Kostianko, {\it  Bi-Lipschitz Man\'e projectors and finite-dimensional reduction
 for complex Ginzburg-Landau equation,} Proc. A  R. Soc.  London, vol. 476, no. 2239, (2020) 1--14.
\bibitem
{KZ15}
A. Kostianko and S. Zelik,
\emph{Inertial manifolds for the 3D Cahn-Hilliard equations with periodic boundary conditions},
 Commun. Pure Appl. Anal., vol. 14, no. 5, (2015) 2069--2094.
\bibitem{KZ18}
A. Kostianko and S. Zelik,
 {\it Inertial manifolds for 1D reaction-diffusion-advection systems. Part II:
Periodic boundary conditions,} Commun. Pure Appl. Anal., vol. 17, no. 1, (2018)  265--317.
\bibitem{KZ17}
 A. Kostianko and S. Zelik,
 {\it Inertial manifolds for 1D reaction-diffusion-advection systems. Part I:
Dirichlet and Neumann boundary conditions,} Commun. Pure Appl. Anal., vol. 16, no. 6, (2017) 2357–
2376.
\bibitem{kwa}
A. Kostianko and S. Zelik, {\it Kwak Transform and Inertial Manifolds revisited}, Jour. Dyn. Diff. Eqns, to appear.
\bibitem
{LS20}
 X. Li and C. Sun,
 \emph{Inertial manifolds for the 3D modified-Leray-$\alpha$ model},
 J. Differential Equations, vol. 268, no. 4, (2020) 1532--1569.
\bibitem
{M-PS88}
J. Mallet-Paret and G. Sell,
\emph{Inertial manifolds for reaction diffusion equations in higher space dimensions},
J. Am. Math. Soc., vol. 1, no. 4, (1988) 805--866.
\bibitem{MPSS93}
J. Mallet-Paret, G. Sell, and Z. Shao, {\it Obstructions to the existence of normally hyperbolic
inertial manifolds,} Indiana Univ. Math. J., 42, no. 3, (1993) 1027–1055.
\bibitem
{M91}
M. Miklavcic,
\emph{A sharp condition for existence of an inertial manifold},
J. Dyn. Differ. Equations, vol. 3, no. 3, (1991) 437--456.
\bibitem
{MZ08}
 A. Miranville and S. Zelik, \emph{Attractors for dissipative partial differential
  equations in bounded and unbounded domains},
   in: Handbook of Differential Equations: Evolutionary Equations,
    vol. IV, Elsevier/North-Holland, Amsterdam, 2008.
\bibitem{Rob1}
J. Robinson, {\it Global Attractors: Topology and Finite-Dimensional Dynamics}, Jour. Dyn.  Dif. Eqns., vol.  11. (1999) 557--581.
\bibitem
{R01}
J. Robinson,
\emph{Infinite-dimensional Dynamical Systems},
Cambridge University Press, 2001.
\bibitem{R11}
J. Robinson, {\it Dimensions, embeddings, and attractors,} Cambridge University Press,
Cambridge, 2011.
\bibitem{faa}
S. Roman, {\it The Formula of Faa di Bruno}, Amer. Math. Monthly,  vol. 87, (1980)  805--809.
\bibitem
{R94}
A.  Romanov, \emph{Sharp estimates for the dimension of inertial manifolds
for nonlinear parabolic equations},
Izv. Math. vol. 43, no. 1, (1994) 31--47.
\bibitem{Rom00}
A. Romanov, {\it Three counterexamples in the theory of inertial manifolds,} Math. Notes, vol. 68,
no. 3–4, (2000) 378--385.
\bibitem{RT96}
R. Rosa and R. Temam,
{\it Inertial manifolds and normal hyperbolicity},
Acta Applicandae Mathematica, vol. 45, (1996) 1--50.
\bibitem
{SY02}
G. Sell and Y. You, \emph{Dynamics of evolutionary equations}, Springer, New York, 2002.
\bibitem{stein}
E. Stein, {\it Singular Integrals and Differentiability Properties of Functions}, Princeton Univ. Press, Princeton, 1970.
\bibitem
{T97}
 R. Temam,
 \emph{Infinite-Dimensional Dynamical systems in Mechanics and Physics},
 second edition, Applied Mathematical Sciences, vol 68, Springer-Verlag, New York, 1997.
\bibitem{wells}
 J.  Wells, {\it Differentiable functions on Banach spaces with Lipschitz
derivatives},  J. Diff. Geom., vol. 8,  (1973) 135--152.
\bibitem
{Z14}
S. Zelik,
\emph{Inertial manifolds and finite-dimensional reduction for dissipative PDEs},
Proc. Royal Soc. Edinburgh 144, vol. 6, (2014) 1245--1327.
\end{thebibliography}
\end{document}